\documentclass[11pt,twoside]{amsart}

\title{The Klain approach to zonal valuations}

\author{Leo Brauner}
\address{Institut f\"ur diskrete Mathematik und Geometrie \newline%
	\indent Technische Universit\"at Wien \newline%
	\indent Wiedner Hauptstrasse 8-10/1047, 1040 Wien, Austria}
\email{leo.brauner@tuwien.ac.at}

\author{Georg C.\ Hofst\"atter}
\address{Institut f\"ur diskrete Mathematik und Geometrie \newline%
	\indent Technische Universit\"at Wien \newline%
	\indent Wiedner Hauptstrasse 8-10/1046, 1040 Wien, Austria}
\email{georg.hofstaetter@tuwien.ac.at}

\author{Oscar Ortega-Moreno}
\address{Institut f\"ur diskrete Mathematik und Geometrie \newline%
	\indent Technische Universit\"at Wien \newline%
	\indent Wiedner Hauptstrasse 8-10/1047, 1040 Wien, Austria}
\email{oscar.moreno@tuwien.ac.at}

\usepackage[english]{babel}
\usepackage[T1]{fontenc}
\usepackage{lmodern}
\usepackage{amsmath,amsthm,amssymb,amsfonts}
\usepackage[abbrev,nobysame]{amsrefs}
\usepackage{float}
\usepackage{hyperref}
\usepackage[capitalise]{cleveref}
\usepackage[left=4cm,right=4cm,top=3cm,bottom=3cm,includeheadfoot]{geometry}
\usepackage{bbm}
\usepackage{xcolor}
\usepackage{enumerate,enumitem}
\usepackage{tikz-cd}
\usepackage{float}
\usepackage{graphicx}
\usepackage{subcaption}
\usepackage{wrapfig}

\hypersetup{colorlinks=true,
	linkcolor=teal,
	citecolor=orange}

\theoremstyle{plain}

\newtheorem{thmintro}{Theorem}
\newtheorem*{thmintro*}{Theorem}
\newtheorem{corintro}[thmintro]{Corollary}
\newtheorem*{corintro*}{Corollary}

\theoremstyle{plain}
\newtheorem{lem}{Lemma}[section]
\newtheorem{prop}[lem]{Proposition}
\newtheorem{thm}[lem]{Theorem}
\newtheorem*{thm*}{Theorem}
\newtheorem{cor}[lem]{Corollary}

\theoremstyle{definition}
\newtheorem{defi}[lem]{Definition}

\theoremstyle{remark}

\newtheorem{rem}[lem]{Remark}

\numberwithin{equation}{section}

\crefname{equation}{}{}
\crefname{lem}{Lemma}{Lemmas}
\crefname{thm}{Theorem}{Theorems}
\crefname{thmintro}{Theorem}{Theorems}
\crefname{rem}{Remark}{Remarks}
\crefname{figure}{Figure}{Figures}

\newcommand{\N}{\mathbb{N}}			% natural numbers
\newcommand{\R}{\mathbb{R}}			% real numbers
\renewcommand{\S}{\mathbb{S}}		% unit sphere
\renewcommand{\H}{\mathbb{H}}		% hemisphere
\newcommand{\Gr}{\mathrm{Gr}}		% Grassmannian
\newcommand{\AGr}{\mathrm{AGr}}		% affine Grassmannian
\newcommand{\BB}{\mathbb{B}}		% ball
\newcommand{\DD}{\mathbb{D}}		% disk
\newcommand{\CC}{\mathcal{C}}		% reference bodies

\newcommand{\K}{\mathcal{K}}		% convex bodies
\newcommand{\Val}{\mathbf{Val}}		% valuations
\newcommand{\MSO}{\mathrm{M}}		% mean section operator

\DeclareMathOperator{\SO}{SO}		% special orthogonal group

\newcommand{\abs}[1]{\lvert #1 \rvert}						% absolute value
\newcommand{\norm}[1]{\lVert #1 \rVert}						% norm
\newcommand{\pair}[2]{\langle #1,#2 \rangle}				% inner product

\newcommand{\Dclass}{\mathcal{D}}	% function classes D^{alpha}

\DeclareMathOperator{\sign}{sgn}

\BibSpec{article}{%
	+{}{\PrintAuthors}  		{author}
	+{,}{ \textit}     		{title}
	+{,}{ }             		{journal}
	+{}{ \textbf}       		{volume}
	+{}{ \parenthesize} 		{date}
	+{}{, no. } 		{number}
	+{,}{ }      	      		{conference}
	+{,}{ }      	      		{book}
	+{,}{ }            		{pages}
	+{,}{ }            	 	{note}
	+{,}{ }            	 	{status}
	+{,}{  \texttt } {eprint}
	+{.}{}              {transition}
}

\BibSpec{book}{%
	+{}{\PrintAuthors}  {author}
	+{,}{ \textit}      {title}
	+{,}{ }             {publisher}
	+{,}{ }             {place}
	+{,}{ }             {date}
	+{.}{}              {transition}
}

\BibSpec{incollection}{%
	+{}  {\PrintAuthors}                {author}
	+{,} { \textit}                     {title}
	+{.} { }                            {part}
	+{:} { \textit}                     {subtitle}
	+{,} { \PrintContributions}         {contribution}
	+{,} { \PrintConference}            {conference}
	+{,} { }                            {booktitle}
	+{,} { pp.~}                        {pages}
	+{,}{ }       		{series}
	+{}{ \textbf}       		{volume}
	+{,} { }                            {publisher}
	+{,} { }                 {date}
	+{,} { }                            {status}
	+{,} { \PrintDOI}                   {doi}
	+{,} { available at \eprint}        {eprint}
	+{}  { \parenthesize}               {language}
	+{}  { \PrintTranslation}           {translation}
	+{;} { \PrintReprint}               {reprint}
	+{.} { }                            {note}
	+{.} {}                             {transition}
}

\begin{document}

	\begin{abstract}

		We show an analogue of the Klain--Schneider theorem for valuations that are invariant under rotations around a fixed axis, called zonal. Using this, we establish a new integral representation of zonal valuations involving mixed area measures with a disk. In our argument, we introduce an easy way to translate between this representation and the one involving area measures, yielding a shorter proof of a recent characterization by Knoerr.
		
		As applications, we obtain various zonal integral geometric formulas, extending results by Hug, Mussnig, and Ulivelli. Finally, we provide a simpler proof of the integral representation of the mean section operators by Goodey and Weil.
	\end{abstract}
	
	\maketitle
	\thispagestyle{empty}
	
	%%%%%%%%%%%%%%%%%%%%%%%%%%%%%%%%%%%%%%%%%%%%%%%%%%%%%%%%%%%%%%%%%%%%%%%%%%%
	\section{Introduction}
	\label{sec:intro}
	%%%%%%%%%%%%%%%%%%%%%%%%%%%%%%%%%%%%%%%%%%%%%%%%%%%%%%%%%%%%%%%%%%%%%%%%%%%

	A \emph{valuation} on the space $\K(\R^n)$ of convex bodies (that is, convex, compact subsets) of $\R^n$, where $n\geq 3$, is a map $\varphi: \K(\R^n) \to \R$  satisfying
	\begin{equation*}
    	\varphi(K) + \varphi(L) = \varphi(K \cup L) + \varphi(K \cap L)
	\end{equation*}
	whenever $K, L \in \K(\R^n)$ and $K \cup L$ is convex. We denote by $\Val(\R^n)$ the space of continuous, translation-invariant valuations on $\R^n$. Valuations in $\Val(\mathbb{R}^n)$ play a central role in convex and integral geometry, appearing naturally in a wide range of applications (see the monographs \cite{Schneider2014,Gardner2006}). Notable examples include the intrinsic volumes -- fundamental geometric quantities that encode information about the size and shape of convex bodies, such as volume, surface area, and mean width -- and, more generally, mixed volumes. 

	Among the most celebrated results in valuation theory is Hadwiger's characterization of rigid motion invariant valuations in terms of intrinsic volumes. This foundational theorem has sparked a long and ongoing line of research (see, e.g., \cite{Bernig2011,Alesker1999,Colesanti2022,Klain1995,Ludwig2010a}).
	Below, we state it in a homogeneous form: for the subspaces $\Val_i(\R^n)\subseteq \Val(\R^n)$ of valuations that are homogeneous of degree $i \in \{0, \dots, n\}$ (that is, $\varphi(\lambda K) = \lambda^i \varphi(K)$ for all $\lambda > 0$).
	
	\begin{thm}[{\cite{Hadwiger1957}}]\label{thm:Hadwiger_intrinsic_vols}
	    For $0 \leq i \leq n$, a valuation $\varphi \in \Val_i(\R^n)$ is rotation invariant if and only if it is a constant multiple of the $i$-th intrinsic volume.
	\end{thm}
	\cref{thm:Hadwiger_intrinsic_vols} was originally proved without assuming homogeneity. However, by McMullen's homogeneous decomposition theorem~\cite{McMullen1980}, the space $\Val(\mathbb{R}^n)$ is the direct sum of the spaces $\Val_i(\mathbb{R}^n)$ for $i \in \{0, \dots, n\}$. Since $\Val_0(\mathbb{R}^n)$ consists only of constant valuations and $\Val_n(\mathbb{R}^n)$ is spanned by the volume functional $V_n$~\cite{Hadwiger1957}, the problem reduces to understanding the intermediate degrees, $1 \leq i \leq n-1$.

	Classification theorems, such as \cref{thm:Hadwiger_intrinsic_vols}, reveal the underlying geometric structure of valuations, providing conceptual proofs of central integral geometric formulas, such as the classical Crofton, Cauchy--Kubota, and kinematic formulas (see, e.g., \cite{Klain1997}). Motivated by this, recent efforts have focused on classification theorems for valuations invariant under different linear groups (see, e.g., \cite{Alesker1999, Alesker2014, Bernig2011,Bernig2011b,Bernig2012,Bernig2017,Bernig2017b,Bernig2014,Bernig2014b,Fu2006,Ludwig2010a,Solanes2017}). 
	One of these is the subgroup $\SO(n-1)$ of $\SO(n)$: the stabilizer of the $n$-th canonical basis vector $e_n$. Valuations invariant under $\SO(n-1)$, called \emph{zonal}, appear naturally in the study of convex-body valued (Minkowski) valuations (see, e.g., \cite{Abardia2011,Dorrek2017,Haberl2012,Kiderlen2006,Ludwig2002,Ludwig2005,Ludwig2010,OrtegaMoreno2021,OrtegaMoreno2023,Brauner2023,Schuster2010,Schuster2015,Schuster2018}) and possess similarities to rigid motion invariant valuations on convex functions (see, e.g.,  \cite{Colesanti2019b,Colesanti2022,Colesanti2022a,Colesanti2023,Colesanti2023a,Knoerr2021,Knoerr2020b}). 
	
	Recently in~\cite{Knoerr2024}, a full characterization of zonal valuations in $\Val(\R^n)$ was obtained using the following function spaces: We set $\Dclass^0=C[-1,1]$ and for $\alpha>0$, we define $\Dclass^\alpha$ as the class of functions $\bar{f}\in C(-1,1)$ such that
	\begin{align*}\label{eq:defDin}
		\lim_{s \to \pm 1}\bar{f}(s)(1-s^2)^{\frac{\alpha}{2}} = 0 \quad\text{and}\quad
		\lim_{s \to \pm 1} \int_0^s \bar{f}(t) (1-t^2)^{\frac{\alpha-2}{2}}dt\text{ exists and is finite}.
	\end{align*}
	
	Knoerr's main result in \cite{Knoerr2024} is the following Hadwiger-type theorem for zonal valuations, which gives an improper integral representation in terms of the $i$-th order \emph{area measure} $S_i(K,{}\cdot{})$ of a convex body $K\in\K(\R^n)$ (cf.~\cite{Schneider2014}*{Ch.~4}). In the following,  a function on $\S^{n-1}$ is called zonal if it is invariant under $\SO(n-1)$.
	
	\begin{thm}[{\cites{Knoerr2024}}]\label{thm:zonalContBallHadwiger}
		For $1\leq i\leq n-1$, a valuation $\varphi \in \Val_i(\R^n)$ is zonal if and only if there exists a function $f=\bar{f}(\pair{e_n}{{}\cdot{}})\in C(\S^{n-1}\setminus\{\pm e_n\})$ with $\bar{f}\in \Dclass^{n-i-1}$ such that
		\begin{equation}\label{eq:zonalContBallHadwiger}
			\varphi(K)
			= \lim_{\varepsilon \to 0^+} \int_{\S^{n-1}\setminus U_{\varepsilon}} f(u) \, dS_i(K,u),
			\qquad K \in \K(\R^n),
		\end{equation}
		where $U_\varepsilon = \{u\in\S^{n-1}: \abs{\pair{e_n}{u}}> 1-\varepsilon\}$. Moreover, $f$ is unique up to the addition of a zonal linear function.
	\end{thm}
	Here, we denote by $\pair{{}\cdot{}}{{}\cdot{}}$ the standard Euclidean inner product on $\R^n$. For degrees $1\leq i< n-1$, \cref{thm:zonalContBallHadwiger} is obtained by approximation from an earlier result of Schuster and Wannerer~\cite{Schuster2018} for \emph{smooth valuations} (see \cref{sec:SmoothHadwiger}). For $i=n-1$, it is an immediate consequence of a classical result by McMullen~\cite{McMullen1980} and the principal value is in fact a proper integral.

	Inspired by the Hadwiger type theorem for convex functions with Monge--Amp\`ere measures~\cites{Colesanti2022a}, our first main result is an analogue of \cref{thm:zonalContBallHadwiger}, where the role of the $i$-th area measure is replaced by the mixed area measure
	\begin{equation*}
		S_i(K,\DD,{}\cdot{})
		= S(K^{[i]},\DD^{[n-i-1]},{}\cdot{}),
	\end{equation*}
	see \cite{Schneider2014}*{Sec.~5.1}, where $\DD$ denotes the $(n-1)$-dimensional unit disk in $e_n^\perp$ and $L^{[j]}$ denotes the tuple consisting of $j$ copies of the body $L$.
	
	\begin{thmintro}\label{thm:zonalDiskHadwiger}
		For $1\leq i\leq n-1$, a valuation $\varphi\in\Val_i(\R^n)$ is zonal if and only if there exists a zonal function $g\in C(\S^{n-1})$ such that
		\begin{equation}\label{eq:zonalDiskHadwiger}
			\varphi(K)
			= \int_{\S^{n-1}} g(u)\, dS_i(K,\DD,u),
			\qquad K\in\K(\R^n).
		\end{equation}
		Moreover, $g$ is unique up to the addition of a zonal linear function.
	\end{thmintro}
	
	This integral representation has certain benefits. In contrast to \cref{eq:zonalContBallHadwiger}, the integral is always proper and the corresponding class of integral kernels is simpler, depending neither on the dimension nor the degree. Moreover, the values of a zonal valuation $\varphi$ on cones with axis $e_n$ determine the integral kernel $g$ (up to linear maps), and therefore, the valuation itself. As a further consequence, we can easily express convergence of zonal valuations as uniform convergence of the integral kernels (see \cref{sec:cones}).

	\subsection*{The Klain approach}
	To establish \cref{thm:zonalDiskHadwiger}, we follow the approach of Klain~\cite{Klain1995}, who provided a significantly simpler proof of Hadwiger's classical \cref{thm:Hadwiger_intrinsic_vols}. At the core of Klain's approach lies the following characterization of valuations vanishing on all hyperplanes, also known as \emph{simple} valuations. Here and throughout, we say that a valuation vanishes on a subspace $E$ if it vanishes on all convex bodies $K \subseteq E$.	
	\begin{thm}[{\cites{Klain1995,Schneider1996}}]\label{thm:klainschneider}
		A valuation $\varphi\in\Val(\R^n)$ vanishes on all hyperplanes if and only if there exist a constant $c\in\R$ and an odd function $f\in C(\S^{n-1})$ such that
		\begin{equation*}
			\varphi(K)
			= cV_n(K) + \int_{\S^{n-1}} f(u)\, dS_{n-1}(K,u),
			\qquad K\in\K(\R^n).
		\end{equation*}
	\end{thm}
	
	Let us point out that \cref{thm:klainschneider} is due to Klain~\cite{Klain1995} for even valuations and due to Schneider~\cite{Schneider1996} for odd valuations. Denoting by $\Gr_k(\R^n)$ the Grassmannian of $k$-dimensional linear subspaces of $\R^n$, a simple corollary of \cref{thm:klainschneider} can be formulated as follows.
	\begin{cor}\label{cor:valDetByRestr}
		Let $0\leq i\leq n-1$ and $\varphi \in \Val_i(\R^n)$. If $\varphi$ vanishes on all subspaces $E \in \Gr_{i+1}(\R^n)$, then $\varphi=0$.
	\end{cor}
	Recently, Klain's approach was employed by Colesanti, Ludwig, and Mussnig in \cite{Colesanti2023a} to give a new and simpler proof of their previously established Hadwiger-type theorem for valuations on convex functions \cite{Colesanti2022}. To adapt Klain's method to our context, we require an analogue of \cref{thm:klainschneider} for zonal valuations.
	
	\begin{thmintro}\label{thm:zonalKlainSchneider}
		A zonal valuation $\varphi \in \Val(\R^n)$ vanishes on some hyperplane containing $e_n$ if and only if there exist a constant $c \in \R$ and a zonal function $f \in C(\S^{n-1})$ vanishing on $\S^{n-1}\cap e_n^{\perp}$ such that
		\begin{align}\label{eq:zonalKlainSchneider}
			\varphi(K)
			= c V_n(K) + \int_{\S^{n-1}} f(u) \,dS_{n-1}(K,u),
			\qquad K \in \K(\R^n).
		\end{align}
	\end{thmintro}
	
	Let us note that a zonal valuation that vanishes on one hyperplane containing $e_n$ \linebreak must already vanish on all such hyperplanes. In a similar way to \cref{thm:klainschneider}, \cref{thm:zonalKlainSchneider} implies that every $i$-homogeneous valuation is already determined by its restriction to a \emph{single} $(i+1)$-dimensional subspace.
	\begin{corintro}\label{cor:zonalValDetByRestr}
		Let $0\leq i\leq n-1$ and $\varphi\in\Val_i(\R^n)$ be zonal. If $\varphi$ vanishes on some subspace $E\in\Gr_{i+1}(\R^n)$ containing $e_n$, then $\varphi= 0$.
	\end{corintro}
	As will be demonstrated later, \cref{cor:zonalValDetByRestr} proves extremely helpful in showing integral geometric formulas for $\SO(n-1)$-invariant quantities. 
	
	Another important step in Klain's approach is the extension of valuations from proper subspaces. In the setting of rigid motion invariant valuations, where only intrinsic volumes appear, this is trivial: the restriction of the $i$-th intrinsic volume to an $i$-dimensional subspace is exactly the volume on that subspace. In general, the problem of extending valuations is more delicate (see, e.g., \cite{Faifman2023}). In fact, for zonal valuations this is not always possible; a difficulty that also arises in the functional setting~\cite{Colesanti2023a}. For smooth valuations, however, we can always extend the integral representations that naturally emerge from Klain's approach.
	\begin{thmintro}\label{thm:ExtendingSmooth}
		Let $1\leq i\leq n-1$ and $E \in \Gr_{i+1}(\R^n)$ be such that $e_n \in E$. Then for every zonal function $f_{\! E}\in C^\infty(\S^{i}(E))$, there exists a zonal function $f\in C^\infty(\S^{n-1})$ such that
		\begin{equation*}
			\int_{\S^{n-1}} f(u) \, dS_i(K,u)
			= \int_{\S^{i}(E)} f_{\! E}(u) \, dS_i^E(K,u),
			\qquad K\in\K(E).
		\end{equation*}
	\end{thmintro}
	
	Here, $\S^i(E)=\S^{n-1} \cap E$  and by $S_i^E(K,{}\cdot{})$ we denote the $i$-th order area measure of $K\subseteq E$ relative to $E \in \Gr_{i+1}(\R^n)$.
	The Hadwiger-type theorem for smooth, zonal valuations of \cite{Schuster2018} is now a direct consequence of \cref{cor:zonalValDetByRestr} and \cref{thm:ExtendingSmooth} (see \cref{sec:SmoothHadwiger}). Let us point out that our proof does not rely on any deep results from representation theory such as the irreducibility theorem~\cite{Alesker2001}.

	\subsection*{Moving between representations}
	
	In order to deduce the general Hadwiger-type theorems (for continuous valuations) from the statement for smooth valuations, it is crucial to understand how to move between the integral representations \cref{eq:zonalContBallHadwiger} and \cref{eq:zonalDiskHadwiger}. Indeed, in \cref{sec:cones}, we define a map $T_{n-i-1}$ that transforms the integral kernel from one representation to the other. In order to show that this is in fact the right transform, by \cref{cor:zonalValDetByRestr}, it suffices to check that the corresponding integral representations coincide on some subspace containing $e_n$. 
	
	These restrictions can be made explicit by certain maps $\pi_{n-i-1}$ and $\pi_{n-i-1,\DD}$, derived from the mixed spherical projections, which were recently introduced by the authors in \cite{Brauner2024} to describe the relations between (mixed) area measures of lower dimensional bodies in different ambient spaces (see \cref{sec:restrictions}).
	
	\begin{figure}[h]
		\begin{tikzcd}[cramped, column sep=tiny,right]
			\Dclass^{n-i-1} \arrow[rd,"\pi_{n-i-1}" below left] \arrow[rr,"T_{n-i-1}","\cong" below] & & C[-1,1] \arrow[ld,"\pi_{n-i-1,\DD}"] \\
			& C[-1,1]
		\end{tikzcd}
		\caption{}\label{fig:commuting_diagram}
	\end{figure}
	Once all the elements are in place, it is easy to check that the diagram in \cref{fig:commuting_diagram} commutes. By \cref{cor:zonalValDetByRestr},	we can thus move between the different representations for zonal valuations via the transform $T_{n-i-1}$. 
	
	Furthermore, using the simpler description of convergence of zonal valuations in terms of the integral representation~\cref{eq:zonalDiskHadwiger}, and the fact that every continuous zonal function trivially defines a valuation in this way, we obtain \cref{thm:zonalDiskHadwiger} by a simple approximation argument. Having established \cref{thm:zonalDiskHadwiger}, we can use the diagram of \cref{fig:commuting_diagram} to recover \cref{thm:zonalContBallHadwiger} without further difficulty.

	\subsection*{Applications}
	
	Just as their counterparts for rigid motion invariant valuations, \cref{thm:zonalKlainSchneider,thm:zonalDiskHadwiger} have various applications to integral geometry that we will now discuss.	
	First, we show the following additive kinematic formula for $\SO(n-1)$, extending a recent result by Hug, Mussnig, and Ulivelli~\cite{Hug2024a}*{Thm.~1.5} from the even to the general case.
	Throughout, we denote by $\kappa_m$ the volume of the $m$-dimensional unit ball in $\R^m$.
	
	\begin{thmintro}\label{thm:zonal_kinematic}
		Let $1\leq j\leq n-1$ and ${g}\in C(\S^{n-1})$ be zonal. For all $K,L\in\K(\R^n)$,
		\begin{align}\label{eq:zonal_kinematic}
			\begin{split}
				& \int_{\SO(n-1)} \int_{\S^{n-1}} g(u) \, dS_j(K+\vartheta L,\DD,u)\, d\vartheta \\
				&\qquad = \frac{1}{\kappa_{n-1}}\sum_{i=0}^j \binom{j}{i} \int_{\S^{n-1}} \int_{\S^{n-1}} \! q(u,v) \, dS_i(K,\DD,u)\, dS_{j-i}(L,\DD,v),				
			\end{split}
		\end{align}
		where $g=\bar{g}(\pair{e_n}{{}\cdot{}})$ and $q(u,v)
		= \max\{\pair{e_n}{u},\pair{e_n}{v}\}\bar{g}(\min\{\pair{e_n}{u},\pair{e_n}{v}\})$.
	\end{thmintro}

	In \cite{Hug2024a}, the statement is derived for even $g$ from an additive kinematic formula for convex functions. As the functional setting is related to the even geometrical setting (see \cref{rem:ksForEvenWoZonal}), this leads to an additional symmetry assumption. Here, as the left-hand side of \cref{eq:zonal_kinematic} is a zonal valuation, \cref{thm:zonalDiskHadwiger} yields \eqref{eq:zonal_kinematic} for an unknown integral kernel $q$, depending a priori on $i$ and $j$. The map $q$ can then be easily determined by plugging in cones with axis $e_n$.
	
	In a similar way, we can recover the following Kubota-type formula from \cite{Hug2024}*{Thm.~3.2}: For $1\leq i\leq n-1$, $K\in\K(\R^n)$, and $f\in C(\S^{n-1})$,
	\begin{align}\label{eq:zonal_Kubota_polarized}
		\int_{\Gr_{i}(\R^n,e_n)} \!\int_{\S^{i-1}(E)} \! f(u)\, dS^E_{i-1}(K|E,u) \, dE
		= \frac{\kappa_{i-1}}{\kappa_{n-1}}\! \int_{\S^{n-1}} \! f(u)\, dS_{i-1}(K,\DD,u).
	\end{align}
	where integration on $\Gr_{i}(\R^n,e_n)=\{E\in\Gr_i(\R^n): e_n\in E\}$ is with respect to the unique rotation invariant probability measure and $K|E$ denotes the orthogonal projection of $K$ onto $E$. From \eqref{eq:zonal_Kubota_polarized}, in turn, we can retrieve the following Crofton-type formula. Here, we denote by $\AGr_{j}(\R^n)$ the affine $j$-Grassmannian, endowed with the rigid motion invariant measure, normalized so that the set of $j$-flats intersecting the unit ball has measure $\kappa_{n-j}$, and by $h_L(u) = \max\{\pair{u}{x}: x \in L\}$ the support function of $L \in \K(\R^n)$.

	\begin{corintro}\label{cor:mean_sec_disk_zonal}
		Let $1\leq j\leq n-1$. If $K\in\K(\R^n)$ is origin-symmetric, then
		\begin{equation}\label{eq:mean_sec_disk_zonal}
			\int_{\AGr_{j}(\R^n)} h_{K\cap E}(e_n) \, dE
			= a_{n,j} V(K^{[n-j+1]},\DD^{[j-1]}),
		\end{equation}
		where $a_{n,j}=\frac{\pi j\kappa_j \kappa_{n-j}}{(j+1)(n-j+1)\kappa_{j+1}\kappa_n }$.
	\end{corintro}
	
	Our primary interest in \cref{eq:mean_sec_disk_zonal} stems from the fact that the expression on the left-hand side coincides with the support function of the mean section body of $K$. Indeed, the $j$-th mean section operator $\MSO_j :\K(\R^n)\to\K(\R^n)$ is defined by
	\begin{equation}\label{eq:defMeanSecOp}
		h_{\MSO_j K}(u)
		= \int_{\AGr_{j}(\R^n)} h_{K\cap E}(u) \, dE,
		\qquad u\in\S^{n-1}.
	\end{equation}
	These operators were introduced by Goodey and Weil~\cite{Goodey1992}, motivated by the question whether every convex body can be reconstructed from the mean of random sections. They gave a positive answer to this question by finding an integral representation of $\MSO_j$ using the Berg functions $g_j\in C^\infty(-1,1)$, $j\geq 2$. The functions $g_j$ were constructed by Berg~\cite{Berg1969} in his solution to the Christoffel problem~\cite{Christoffel1865} so that for every dimension $n\geq 2$ and $K \in \K(\R^n)$,
	\begin{equation}\label{eq:Bergs_result}
		h_{K-s(K)}(u) = \int_{\S^{n-1}} g_n(\pair{u}{v})\, dS_1(K,v),
		\qquad u\in\S^{n-1},
	\end{equation}
	where $s(K)$ denotes the \emph{Steiner point} of $K$ (cf.~\cite{Schneider2014}*{p.~50}). Interestingly enough, the integral representation of $\MSO_j$ of type \cref{eq:zonalContBallHadwiger} arises by lifting Berg's functions to the unit sphere in different dimensions.
	
	\begin{thm}[{\cites{Goodey1992,Goodey2014}}]\label{MSO_Berg_fct}
		Let $2\leq j< n$. Then for every $K\in\K(\R^n)$,
		\begin{equation}\label{eq:MSO_Berg_fct}
			h_{\MSO_j (K-s(K))}(u)
			= %\frac{j \kappa_j\kappa_{n-j}}{(n-j+1)n \kappa_n}
			c_{n,j} \int_{\S^{n-1}} g_j(\pair{u}{v})\, dS_{n-j+1}(K,v),
			\qquad u\in\S^{n-1},
		\end{equation}
		where $c_{n,j}=\frac{j \kappa_j\kappa_{n-j}}{(n-j+1)n \kappa_n}$.
	\end{thm}
	Let us note that the case $j=2$ of \cref{MSO_Berg_fct} was settled in \cite{Goodey1992}, while the cases $2 <j<n$ were deduced from this more recently in \cite{Goodey2014}. The proofs in \cite{Goodey1992,Goodey2014} rely on heavy tools from harmonic analysis. Applying our \cref{cor:zonalValDetByRestr}, we can give a new, shorter proof of the results in \cite{Goodey2014} using the case $j=2$ from \cite{Goodey1992}.
	\begin{corintro}\label{cor:MSOcase2implJ}
		\cref{MSO_Berg_fct} holds for all $2 <j<n$.
	\end{corintro}

	\subsection*{Organization of the article}

	In \cref{sec:moving_between_integral_reps}, we introduce the transform $T_{n-i-1}$ and examine restrictions of integral representations; in there, we establish the commuting diagram above and \cref{thm:ExtendingSmooth}. In \cref{sec:zonalKlainSchneider}, we prove \cref{thm:zonalKlainSchneider} and the subsequent corollary. \cref{sec:zonalHadwigerThms} is devoted to the Hadwiger type theorems for zonal valuations; using our findings from the previous sections, we show \cref{thm:zonalContBallHadwiger} and \cref{thm:zonalDiskHadwiger}. Finally, in \cref{sec:applications}, we discuss the applications to integral geometry and the mean section operators.
	
	%%%%%%%%%%%%%%%%%%%%%%%%%%%%%%%%%%%%%%%%%%%%%%%%%%%%%%%%%%%%%%%%%%%%%%%%%%%
	\section{Moving between integral representations}
	\label{sec:moving_between_integral_reps}
	%%%%%%%%%%%%%%%%%%%%%%%%%%%%%%%%%%%%%%%%%%%%%%%%%%%%%%%%%%%%%%%%%%%%%%%%%%%
	
	In this section, we investigate how we can move between different integral representations of zonal valuations in $\Val_i(\R^n)$ -- in terms of the $i$-th area measure and the mixed area measure with the disk -- and how to restrict to and extend from an $(i+1)$-dimensional subspace containing $e_n$. For $1\leq i\leq n-1$ and zonal functions $f,g\in C(\S^{n-1})$, we define zonal valuations $\varphi_{i,f},\psi_{i,g}\in\Val_i(\R^n)$ by
	\begin{equation*}
		\varphi_{i,f}(K)
		:= \int_{\S^{n-1}} \! f(u) \, dS_i(K,u)
		\qquad\text{and}\qquad
		\psi_{i,g}(K)
		:= \int_{\S^{n-1}}  \! g(u) \, dS_i(K,\DD,u)
	\end{equation*}
	for $K\in\K(\R^n)$. We want to find a transform $T_{n-i-1}$ with the property that $\varphi_{i,f}=\psi_{i,g}$, whenever $T_{n-i-1}\bar{f}=\bar{g}$, where $f=\bar{f}(\pair{e_n}{{}\cdot{}})$ and $g=\bar{g}(\pair{e_n}{{}\cdot{}})$.
	We will evaluate $\varphi_{i,f}$ and $\psi_{i,g}$ on a certain family of cones to see how $T_{n-i-1}$ needs to be defined. Then, we show that this transform ensures that $\varphi_{i,f}$ and $\psi_{i,g}$ also coincide on subspaces $E\in\Gr_{i+1}(\R^n)$ containing $e_n$.	

	\subsection{Evaluation on cones}
	\label{sec:cones}
	%%%%%%%%%%%%%%%%%%%%%%%%%%%%%%%%%%%%%%%%%%%%%%%%%%%%%%%%%%%%%%%%%%%%%%%%%%%
	
	For $s \in [-1,1] \setminus\{0\}$, we denote by $C_s$ the cone with basis $\DD$ and apex $\frac{\sqrt{1-s^2}}{s}e_n$, that is,
	\begin{align*}
		C_s := \mathrm{conv}\!\left(\DD \cup \left\{\tfrac{\sqrt{1-s^2}}{s}e_n\right\}\right).
	\end{align*}
	Observe that $C_{-s}=-C_s$; for $s>0$, the cone $C_s$ is pointing ``up'', for $s<0$, it is pointing ``down''. As $s\to 0$, the height of $C_s$ tends to infinity, and $C_{-1}=C_1=\DD$. Moreover, the support function of $C_s$ is given by
	\begin{align*}
		h_{C_s}(u)
		%= \max\left\{\sqrt{1-\pair{\pair}{u}^2}, \frac{\sqrt{1-s^2}}{s}\pair{e_n}{u}\right\} =
		= \begin{cases}
			\sqrt{1-\pair{e_n}{u}^2},				& \sign(s)\pair{e_n}{u} \leq \abs{s},\\
			\frac{\sqrt{1-s^2}}{s}\pair{e_n}{u},	& \sign(s)\pair{e_n}{u} \geq \abs{s},
		\end{cases}
		\qquad u\in\S^{n-1}.
	\end{align*}
	
	Evaluating $\varphi_{i,f}$ on the cones $C_s$ boils down to computing their area measures. This has been done recently in \cite{Knoerr2024}*{Lemma~2.2}.

	\begin{lem}[{\cite{Knoerr2024}}]	\label{lem:area_meas_cone_ball}
		Let $1\leq i<n-1$ and $f=\bar{f}(\pair{e_n}{{}\cdot{}})\in C(\S^{n-1})$. For $s\in [-1,1]\setminus\{0\}$,
		\begin{equation*}
			\varphi_{i,f}(C_s)
			= \kappa_{n-1} \bigg(\frac{1}{\abs{s}}(1-s^2)^{\frac{n-i-1}{2}}\bar{f}(s) + (n-i-1) \sign{s} \int_{-\sign{s}}^s \bar{f}(t)(1-t^2)^{\frac{n-i-3}{2}}\! dt  \bigg).
		\end{equation*}
%		\begin{equation*}
%			\begin{split}
%				&\varphi_{i,f}(C_s)
%				= \kappa_{n-1} \bigg(\frac{1}{\abs{s}}(1-s^2)^{\frac{n-i-1}{2}}\bar{f}(s) \\
%				&\qquad\quad + (n-i-1) \sign{s} \int_{-\sign{s}}^s \bar{f}(t)(1-t^2)^{\frac{n-i-3}{2}}\! dt  \bigg),
%				\quad s\in [-1,1]\setminus\{0\}.
%			\end{split}				
%		\end{equation*}
	\end{lem}
	
%	The outer unit normals of $C_s$ are given by $-e_n$ (the basis) and $\{u \in \S^{n-1}: \pair{e_n}{u} = s\}$, excluding here normal vectors at the apex and the boundary of the basis.
	
	Next, we evaluate the valuation $\psi_{i,g}$ on the family $C_s$.
	
	\begin{lem}	\label{lem:area_meas_cone_disk}
		Let $1\leq i\leq n-1$ and $g=\bar{g}(\pair{e_n}{{}\cdot{}})\in C(\S^{n-1})$. For $s\in[-1,1]\setminus\{0\}$,
		\begin{equation}\label{eq:area_meas_cone_disk}
			\psi_{i,g}(C_s)
			= \kappa_{n-1}\left(\bar{g}(-\sign{s}) + \frac{1}{\abs{s}}\bar{g}(s) \right).
		\end{equation}
	\end{lem}
	\begin{proof}
		First, note that $C_{-s}=-C_s$, and thus, the case when $s<0$ can easily be deduced from the case when $s>0$.
		Therefore, we will now restrict ourselves to the case when $s>0$. For degree $i=n-1$,  we deduce \cref{eq:area_meas_cone_disk} from the classical facts that the surface area measure $S_{n-1}(C_s,{}\cdot{})$ is the area of the reverse spherical image of $C_s$ and that it integrates all linear functions to zero (cf.~\cite{Schneider2014}*{Section~4.2}). 
		
		For degrees $1\leq i< n-1$, note that for $\lambda \geq 0$, the body $\lambda C_s+\DD$ is a truncated cone with basis $(\lambda+1)\DD$ which is cut off at a unit disk of radius one. That is,
		\begin{equation*}
			\lambda C_s + \DD
			= \big((\lambda+1)C_s \setminus (C_s + \tfrac{\sqrt{1-s^2}}{s}e_n)\big) \cup (\DD + \tfrac{\sqrt{1-s^2}}{s}e_n).
		\end{equation*}
%		\begin{equation*}
%			(\lambda+1)C_s \cup (C_s + \tfrac{\sqrt{1-s^2}}{s}e_n) = C_s
%			\quad\text{and}\quad
%			(\lambda+1)C_s \cup (C_s + \tfrac{\sqrt{1-s^2}}{s}e_n) = \DD + \tfrac{\sqrt{1-s^2}}{s}e_n
%		\end{equation*}
		Hence, by the valuation property and the translation invariance of the surface area measure, we have that 
		\begin{equation*}
			S_{n-1}(\lambda C_s+\DD,{}\cdot{})
			= S_{n-1}((\lambda+1)C_s,{}\cdot{}) - S_{n-1}( C_s,{}\cdot{}) + S_{n-1}(\DD,{}\cdot{}).
		\end{equation*}
		Subsequently, applying \cref{eq:area_meas_cone_disk} for degree $n-1$, we obtain that
		\begin{align*}
			\psi_{n-1,g}(\lambda C_s + \DD)
			= \kappa_{n-1}\bigg(((\lambda+1)^{n-1} - 1)\bigg(\bar{g}(-1) + \frac{1}{s}\bar{g}(s)\bigg) + (\bar{g}(-1)+\bar{g}(1))\bigg).
		\end{align*}
		Moreover, by the multilinearity of the surface area measure (cf.~\cite{Schneider2014}*{p.~280}),
		\begin{equation*}
			\psi_{n-1,g}(\lambda C_s + \DD)
			= \sum_{i=0}^{n-1} \binom{n-1}{i} \lambda^{i} \psi_{i,g}(C_s).
		\end{equation*}
		By a comparison of coefficients, we obtain \cref{eq:area_meas_cone_disk} for degrees $1\leq i< n-1$.
	\end{proof}
	
	 Note that \cref{eq:area_meas_cone_disk} allows us to extract the function $g$ easily from the valuation $\psi_{i,g}$. To pursue this further, for $1\leq i\leq n-1$ and $\varphi\in\Val_i(\R^n)$, we define  a function $\bar{g}_{\varphi}:[-1,1]\to \R$ by
	\begin{equation}\label{eq:g_extracted}
		\bar{g}_{\varphi}(s)
		:= \frac{1}{\kappa_{n-1}} \cdot\begin{cases}
			s(\varphi(\DD)-\varphi(C_s)),						&s\in[-1,0), \\
			\frac{1}{i}(\varphi(\DD+[0,e_n]) - \varphi(\DD)),	&s=0, \\
			s\varphi(C_s),										&s \in (0,1].
		\end{cases}
	\end{equation}	
	It is initially unclear whether $\bar{g}_{\varphi}$ is continuous at $s=0$. This will be ensured by the following elementary lemma; it was shown in \cite{Knoerr2024}*{Proposition~4.4}, but we also provide a proof for the convenience of the reader. Here and in the following, we denote by $\BB$ the unit ball in $\R^n$ and by $\norm{\varphi}:=\sup\{\abs{\varphi(K)}:K\subseteq \BB\}$ the Banach norm on $\Val(\R^n)$.
	
	\begin{lem}[{\cite{Knoerr2024}}]\label{lem:sphiC_s}
		Let $1\leq i\leq n-1$ and $\varphi\in\Val_i(\R^n)$. Then $\bar{g}_{\varphi}\in C[-1,1]$ and
		$\norm{\bar{g}_{\varphi}}_{\infty}\leq M_{i}\norm{\varphi}$, where $M_{i}>0$ is a constant depending only on $i$.
	\end{lem}
	\begin{proof}
		Observe that for $s\in (0,\frac{1}{\sqrt{2}})$, the cone $C_s$ has height strictly larger than one, so making a cut at height one splits it into two parts. The upper part is the cone $a_sC_s+e_n$ and the lower part is the truncated cone $Z_s=\mathrm{conv}(\DD\cup (a_s\DD + e_n))$, where we defined $a_s = 1-s/\sqrt{1-s^2}$. That is,
%		\begin{align*}
%			Z_s = (C_s \setminus (a_sC_s+ e_n)) \cup (a_s\DD + e_n)
%		\end{align*}
		\begin{align*}
			Z_s \cup (a_sC_s + e_n) = C_s
			\qquad\text{and}\qquad
			Z_s \cap (a_sC_s + e_n) = a_s\DD + e_n.
		\end{align*}
		Hence, by the valuation property, translation invariance, and homogeneity of $\varphi$,
		\begin{equation*}
			s\varphi(C_s)
			= \frac{s}{1-a_s^i} (\varphi(C_s) - \varphi(a_sC_s))
			= \frac{s}{1-a_s^i}(\varphi(Z_s) - \varphi(a_s\DD)).
		\end{equation*}
		First, note that if we pass to the limit $s\to 0^+$, then $a_s\to 1$, and by L'Hôpital's rule, $s/(1-a_s^i)\to 1/i$. Moreover $Z_s$ converges to the cylinder $Z_0=\DD+[0,e_n]$, so by the continuity of $\varphi$, the right hand side converges to $\frac{1}{i}(\varphi(Z_0)-\varphi(\DD))$. Repeating this argument for $s\in (-\frac{1}{\sqrt{2}},0)$ yields the first claim.
		
		For the second claim, note that for $s\in (0,\frac{1}{\sqrt{2}})$, the term $s/(1-a_s^i)$ is bounded by some number $M_i'>0$. Thus,
		\begin{equation*}
			\abs{s\varphi(C_s)}
			\leq M_i'(\abs{\varphi(Z_s)}+\abs{\varphi(a_s\DD)})
			\leq M_i'(\sqrt{2}^{\, i} + 1)\norm{\varphi},
		\end{equation*}
		due to the fact that $a_s\DD\subseteq \BB$ and $Z_s\subseteq Z_0\subseteq\sqrt{2}\BB$. If $s\in [\frac{1}{\sqrt{2}},1]$, then $C_s\subseteq \BB$, and thus, $\abs{s\varphi(C_s)}\leq \norm{\varphi}$. Repeating this argument for $s\in[-1,0)$ yields the second claim.
	\end{proof}
	
	Note that the lemma above does not require the valuation $\varphi$ to be zonal.	
	For the converse estimate on zonal valuations, let $1\leq i\leq n-1$ and $g\in C(\S^{n-1})$ be zonal. For every $K\in\K(\R^n)$ such that $K\subseteq \BB$,
	\begin{align}\label{eq:psi_norm_estimate}
		\begin{split}
			&\psi_{i,g}(K)
			= \int_{\S^{n-1}} g(u) \, dS_i(K,\DD,u)
			\leq \int_{\S^{n-1}} \, dS_i(K,\DD,u)  \, \norm{g}_\infty \\
			&\qquad = nV(K^{[i]},\DD^{[n-i-1]},\BB) \norm{g}_{\infty} 
			\leq nV_n(\BB) \norm{g}_\infty,
			%= \omega_n \norm{g}_\infty,
		\end{split}		
	\end{align}
	which shows that $\norm{\psi_{i,g}}\leq n \kappa_n \norm{g}_\infty$. 
	From \cref{thm:zonalDiskHadwiger}, we will obtain that every zonal valuation $\varphi\in\Val_i(\R^n)$ is of the form $\varphi=\psi_{i,g}$, where we can choose $g$ to be $g=\bar{g}_{\varphi}(\pair{e_n}{{}\cdot{}})$. From this, we will deduce that zonal valuations are determined on cones, and that convergence in the Banach norm is equivalent to uniform convergence of the corresponding integral kernels (see \cref{sec:ContHadwiger}).

	%%%%%%%%%%%%%%%%%%%%%%%%%%%%%%%%%%%%%%%%%%%%%%%%%%%%%%%%%%%%%%%%%%%%%%%%%%%
	\subsection{Restricting to subspaces}
	\label{sec:restrictions}
	%%%%%%%%%%%%%%%%%%%%%%%%%%%%%%%%%%%%%%%%%%%%%%%%%%%%%%%%%%%%%%%%%%%%%%%%%%%
		
	Now we investigate how the integral representations of $\varphi_{i,f}$ and $\psi_{i,g}$ behave when restricted to subspaces containing $e_n$. 
	In the following, for $E\in\Gr_k(\R^n)$ and $u\in\S^{k-1}(E)$, we define the relatively open $(n-k)$-dimensional half-sphere generated by $E^\perp$ and $u$ as
	\begin{align*}
		\H^{n-k}(E,u)
		&= \{ v\in\S^{n-1}\setminus E^\perp: (v|E)/\norm{v|E} =u \} \\
		&= \{ v\in\S^{n-k}(E^\perp \vee u): \pair{u}{v}>0\}.
	\end{align*}
	Here, $E^\perp \vee u=\mathrm{span}(E^\perp\cup u)$ denotes the subspace generated by $E^\perp$ and $u$. When dealing with mixed area measures of lower dimensional bodies, a key tool will be provided by the mixed spherical projections and liftings that were introduced recently by the authors~\cite{Brauner2024}.
	
	\begin{defi}[{\cite{Brauner2024}*{Definition~2.3}}]
		Let $1\leq k< n$ and $E\in\Gr_{k}(\R^n)$. Also, let $C_1,\ldots ,C_{n-k}\in\K(\R^n)$ and set $\CC=(C_1,\ldots,C_{n-k})$. The \emph{$\CC$-mixed spherical projection} is the bounded linear operator $\pi_{E,\CC}: C(\S^{n-1}) \to C(\S^{k-1}(E))$,
		\begin{equation*}
			(\pi_{E,\CC}f)(u)
			= \int_{\H^{n-k}(E,u)} f(v)\, dS^{E^\perp \vee u}(\CC| (E^\perp \vee u),v),
			\qquad u\in\S^{k-1}(E).
		\end{equation*}
		We call its adjoint operator $\pi_{E,\CC}^\ast: \mathcal{M}(\S^{k-1}(E))\to\mathcal{M}(\S^{n-1})$ the \emph{$\CC$-mixed spherical lifting}. That is, for $\mu\in\mathcal{M}(\S^{k-1}(E))$ and $f\in C(\S^{n-1})$,
		\begin{equation*}
			\int_{\S^{n-1}} f \, d(\pi_{E,\CC}^\ast \mu)
			= \int_{\S^{k-1}(E)} \pi_{E,\CC}f \, d\mu.
		\end{equation*} 
	\end{defi}
	
	Here, we used the abbreviation $\CC|E'=(C_1|E',\ldots,C_{n-k}|E')$. Moreover, in the case when $C_1=\cdots=C_{n-k}=C$, we will write $\pi_{E,\CC}=\pi_{E,C}$. In~\cite{Brauner2024}, the authors established the following theorem, expressing the mixed area measure of a lower dimensional body in terms of its surface area measure relative to a subspace.
	
	\begin{thm}[{\cite{Brauner2024}*{Theorem~B}}]\label{thm:mixed_area_meas_liftSmooth}
		Let $1\leq i< n-1$ and $E\in\Gr_{i+1}(\R^n)$. Also, let $\CC=(C_1,\ldots,C_{n-i-1})$ be a family of convex bodies with $C^2$~support functions. Then for all $K\in\K(E)$,
		\begin{equation}\label{eq:mixed_area_meas_lift:intro}
			S(K^{[i]},\CC,{}\cdot{})
			= \frac{1}{\binom{n-1}{i}} \pi_{E,\CC}^\ast S_i^E(K,{}\cdot{}).
		\end{equation}
	\end{thm}
	
	In the instance where the reference bodies $C_1,\ldots,C_{n-i-1}$ are Euclidean balls, this coincides with a particular case of a result by Goodey, Kiderlen, and Weil (see \cite{Goodey2011}*{Theorem~6.2}).
	In order to compute the $\BB$-mixed and $\DD$-mixed spherical projection of zonal functions, we will need spherical cylinder coordinates: For every $\bar{f}\in C[-1,1]$,	
	\begin{equation*}
		\int_{\S^{n-1}} \bar{f}(\pair{e_n}{u}) \, du
		= \omega_{n-1} \int_{-1}^{1} \bar{f}(t) (1-t^2)^{\frac{n-3}{2}} \, dt,
	\end{equation*}
	where $\omega_{m}$ denotes the surface area of the unit sphere in $\R^m$ (cf.~\cite{Groemer1996}*{p.~9}). In the following, we also define $\omega_{\alpha}:=2\pi^{\frac{\alpha}{2}}/\Gamma(\frac{\alpha}{2})$ for $\alpha>0$.

	\begin{lem} \label{sph_proj_zonal}
		Let $1\leq i< n-1$ and $E\in\Gr_{i+1}(\R^n)$ be such that $e_n\in E$.
		Then for every function $f=\bar{f}(\pair{e_n}{{}\cdot{}})\in C(\S^{n-1})$, we have $\pi_{E,\BB}f=(\pi_{n-i-1,\BB}\bar{f})(\pair{e_n}{{}\cdot{}})$, where we define for $\alpha>0$,
		\begin{equation*} \label{eq:sph_proj_zonal}
			(\pi_{\alpha,\BB}\bar{f})(s)
			= \omega_{\alpha} \int_0^1 \bar{f}(st) (1-t^2)^{\frac{\alpha-2}{2}}dt,
			\qquad s\in (-1,1).
		\end{equation*}
	\end{lem}
	\begin{proof}
		By definition of the mixed spherical projection,
		\begin{equation*}
			(\pi_{E,\BB}f)(u)
			= \int_{\H^{n-i-1}(E,u)} \bar{f}(\pair{e_n}{v})\, dv
			= \int_{\H^{n-i-1}(E,u)} \bar{f}(s\pair{u}{v})\, dv,
		\end{equation*}
		where $dv$ denotes the spherical Lebesgue measure on $\S^{n-i-1}(E\vee u^\perp)$ and the second equality is due to the fact that $\pair{e_n}{v}=\pair{e_n|(E^\perp\vee u)}{v}=\pair{e_n}{u}\pair{u}{v}$.
		Applying spherical cylinder coordinates in $\S^{n-i-1}(E^\perp\vee u)$ then yields the desired identity.
	\end{proof}
	
	Next, we want to do the same for $\pi_{E,\DD}$. However, the disk $\DD$ does not have a $C^2$~support function, so we can not apply \cref{thm:mixed_area_meas_liftSmooth} directly. By an approximation argument, we can obtain the required formula as a corollary of \cref{thm:mixed_area_meas_liftSmooth}. For this, we recall the classical Portmanteau theorem.
	
	\begin{thm}[{\cite{Klenke2020}*{Theorem~13.16}}]\label{Portmanteau}
		Let $\mu_k,\mu$ be finite positive measures on a compact metric space $X$. Then the following are equivalent:
		\begin{enumerate}[label=\upshape(\alph*)]
			\item \label{Portmanteau:weak_conv}
			$\mu_k \to \mu$ weakly.
			\item \label{Portmanteau:cont_fct}
			For every $f\in C(X)$, we have $\lim_{k\to\infty} \int_X f d\mu_k	= \int_X f d\mu$.
			\item \label{Portmanteau:bounded_fct}
			For every bounded, measurable function $f$ on $X$ such that its discontinuity points are a set of $\mu$-measure zero, $\lim_{k\to\infty} \int_X f d\mu_k	= \int_X f d\mu$.
		\end{enumerate}
	\end{thm}
	
	\begin{cor}\label{thm:liftAreaMeasDisk}
		Let $1\leq i< n-1$ and $E\in\Gr_{i+1}(\R^n)$ be such that $E \not \subseteq e_n^\perp$. Then for all $K\in\K(E)$,
		\begin{equation}
			S(K^{[i]},\DD^{[n-i-1]},{}\cdot{})
			= \frac{1}{\binom{n-1}{i}} \pi_{E,\DD}^\ast S_i^E(K,{}\cdot{}).
		\end{equation}
	\end{cor}
	\begin{proof}
		Take a sequence of convex bodies $D_k \in \K(\R^n)$ with $C^2$~support functions, converging to $\DD$ in the Hausdorff metric. By \cref{thm:mixed_area_meas_liftSmooth}, for every $f\in C(\S^{n-1})$,
		\begin{equation*}
			\int_{\S^{n-1}}f(u)\, dS(K^{[i]},D_k^{[n-i-1]},u)
			= \frac{1}{\binom{n-1}{i}} \int_{\S^{i}(E)} (\pi_{E,D_k}f)(u) \, dS_i^E(K,u).
		\end{equation*}
		We want to pass to the limit $k\to\infty$ on both sides. On the left hand side, as mixed area measures are weakly continuous,
		\begin{align*}
			\int_{\S^{n-1}} f(u) \, dS(K^{[i]}, D_k^{[n-i-1]},u) \to \int_{\S^{n-1}} f(u) \, dS(K^{[i]}, \DD^{[n-i-1]}, u).
		\end{align*}
		Next we want to establish pointwise convergence of $\pi_{E,D_k}f$ to $\pi_{E,\DD}f$. To this end, we first show that for all $u\in \S^{i}(E)$,
		\begin{align*}\label{eq:prfMixedAreaLiftDiskNullmeas}
			S_{n-i-1}^{E^\perp \vee u}(\DD| (E^\perp \vee u), \S^{n-i-1}(E^\perp \vee u) \cap E^\perp) = 0.
		\end{align*}
		We consider the surface area measure $S_{n-i-1}^F(\DD|F,{}\cdot{})$ on the unit sphere $\S^{n-i-1}(F)$ of $F=E^\perp \vee u$, and distinguish two cases. If $e_n\notin F$, then $\DD|F$ is a smooth convex body in $F$, so $S_{n-i-1}^F(\DD|F,{}\cdot{})$ is absolutely continuous with respect to the spherical Lebesgue measure. If $e_n\in F$, then $\DD|F$ is a disk in $F$, and thus, $S_{n-i-1}^F(\DD|F,{}\cdot{})$ is concentrated on the two points $e_n$ and $-e_n$. In either case, the set $\S^{i}(F)\cap E^\perp$, which is a great sphere of $\S^{i}(F)$ containing neither $e_n$ nor $-e_n$, is a null set of $S_{n-i-1}^F(\DD|F,{}\cdot{})$.
		
		Consequently, by the definition of the mixed spherical projection, \cref{Portmanteau}, and the fact that the set $\S^{n-i-1}(E^\perp \vee u) \cap E^\perp$ contains all possible discontinuity points of $\mathbbm{1}_{\H^{n-i-1}(E,u)} f$, we obtain that $(\pi_{E,D_k}f) (u) \to (\pi_{E,\DD} f) (u)$ for all $u \in \S(E)$. By dominated convergence,
		\begin{align*}
			\int_{\S^{i}(E)} (\pi_{E,D_k} f) (u) \, dS_i^E(K,u) \to \int_{\S^{i}(E)}(\pi_{E,\DD} f) (u)\, dS_i^E(K,u),
		\end{align*}
		which yields the desired identity.	
	\end{proof}
	
	\begin{rem}
		Note that the proof of \cref{thm:liftAreaMeasDisk} works verbatim if the disk $\DD$ is replaced by any smooth convex body in $e_n^\perp$.
		
		We also want to comment on the condition that $E\not \subseteq e_n^\perp$. Since we mainly consider restrictions so subspaces containing $e_n$, this is not an obstacle to our purposes.
		However, we want to point out that the condition is necessary. Indeed, if $E\subseteq e_n^\perp$, then for all $K\in\K(E)$,
		\begin{equation*}
			S(K^{[i]},\DD^{[n-i-1]},{}\cdot{}) = \frac{\kappa_{n-1}\kappa_{n-i-1}}{\binom{n-1}{i}}V_i(K)\left(\delta_{e_n} + \delta_{-e_n}\right).
		\end{equation*}
		This follows by polarization from the fact that $S_{n-1}(C,{}\cdot{})=V_{n-1}(C)(\delta_{e_n}+\delta_{-e_n})$ for every convex body $C\in\K(e_n^\perp)$. This also exemplifies that the regularity condition in \cref{thm:mixed_area_meas_liftSmooth} can not be dropped completely. 
	\end{rem}
	
	Next, we prove the analogue of \cref{sph_proj_zonal} for the disk. For this, we need the following formula of the surface area measure of smooth convex bodies of revolution. It is an easy consequence of the proof of \cite{OrtegaMoreno2021}*{Lemma~5.3}.
	
	\begin{lem}[\cite{OrtegaMoreno2021}]\label{lem:surf_area_meas_body_of_revol}
		Let $L\in\K(\R^n)$ be a convex body of revolution with support function $h_L=\eta(\pair{e_n}{{}\cdot{}}) \in C^2(\S^{n-1})$. Then
		\begin{equation}\label{eq:surf_area_meas_body_of_revol}
			dS_{n-1}(L,u)
			= (\mathcal{A}_1 \eta)(\pair{e_n}{u})^{n-2}(\mathcal{A}_2\eta)(\pair{e_n}{u}) du,
		\end{equation}
		where $(\mathcal{A}_1 \eta)(t) = \eta(t) - t \eta'(t)$ and $(\mathcal{A}_2 \eta)(t) = (1-t^2)\eta''(t) + \eta(t) - t \eta'(t)$.
	\end{lem}
	
	\begin{lem}\label{mixed_sph_proj_zonal_disk}
		Let $0\leq i<n-1$ and $E\in\Gr_{i+1}(\R^n)$ be such that $e_n\in E$.
		Then for every $f=\bar{f}(\pair{e_n}{{}\cdot{}})\in C(\S^{n-1})$, we have that $\pi_{E,\DD}f=(\pi_{n-i-1,\DD}\bar{f})(\pair{e_n}{{}\cdot{}})$, where we define for $\alpha>0$, 
		\begin{equation*} \label{eq:mixed_sph_proj_zonal_disk}
			(\pi_{\alpha,\DD}\bar{f})(s)
			= \omega_{\alpha} (1-s^2) \int_0^1 \bar{f}(st)(1-s^2t^2)^{-\frac{\alpha+2}{2}} (1-t^2)^{\frac{\alpha-2}{2}} dt,
			\qquad s\in (-1,1).
		\end{equation*}
	\end{lem}
	\begin{proof}
		By definition of the mixed spherical projection, for all $u\in\S^{i}(E)$,
		\begin{align*}
			(\pi_{E,\DD}f)(u)
			= \int_{\H^{n-i-1}(E,u)} \bar{f}(\pair{e_n}{v}) \, dS_{n-i-1}^{E^\perp \vee u}(\DD | (E^\perp \vee u), v).
		\end{align*}
		Note that for all $v\in\S^{n-i-1}(E^\perp\vee u)$, we have that for $s = \pair{e_n}{u}$,
		\begin{equation*}
			h_{\DD|(E^\perp\vee u)}(v)
			= h_{\DD}(v|(E^\perp\vee u))
			= h_{\DD}(v)
			= \sqrt{1-\pair{e_n}{v}^2}
			= \sqrt{1-s^2\pair{u}{v}^2},
		\end{equation*}
		since $\pair{e_n}{v}= \pair{e_n|(E^\perp\vee u)}{v}= \pair{e_n}{u}\pair{u}{v}$. Consequently, the projected disk $\DD|(E^\perp\vee u)$ is a smooth body of revolution in $E^\perp\vee u$ with axis of revolution~$u$. Its support function is given by $h_{\DD|(E^\perp\vee u)}= \eta_s(\pair{u}{{}\cdot{}})$, where $\eta_s(t)=\sqrt{1-s^2t^2}$. Direct computation shows that	
		\begin{equation*}
			\mathcal{A}_1\eta_s(t)=(1-s^2t^2)^{-\frac{1}{2}}
			\qquad\text{and}\qquad
			\mathcal{A}_2\eta_s(t)=(1-s^2)(1-s^2t^2)^{-\frac{3}{2}}.
		\end{equation*}
		Therefore, by \cref{eq:surf_area_meas_body_of_revol}, we obtain that
		\begin{equation*}
			(\pi_{E,\DD}f)(u)
			= (1-s^2) \int_{\H^{n-i-1}(E,u)} \bar{f}(s\pair{u}{v})(1-s^2\pair{u}{v}^2)^{-\frac{n-i+1}{2}} dv.
		\end{equation*}
		Applying spherical cylinder coordinates in $\S^{n-i-1}(E^\perp\vee u)$ then yields the desired identity.
	\end{proof}

	%%%%%%%%%%%%%%%%%%%%%%%%%%%%%%%%%%%%%%%%%%%%%%%%%%%%%%%%%%%%%%%%%%%%%%%%%%%
	\subsection{The commuting diagram}
	\label{sec:commuting_diagram}
	%%%%%%%%%%%%%%%%%%%%%%%%%%%%%%%%%%%%%%%%%%%%%%%%%%%%%%%%%%%%%%%%%%%%%%%%%%%

	Comparing the expressions for $\varphi_{i,f}(C_s)$ found in \cref{lem:area_meas_cone_ball} and for $\psi_{i,g}(C_s)$ found in \cref{lem:area_meas_cone_disk} motivates the following definition of a family of integral transforms. 
	
	\begin{defi}\label{defi:T_alpha}
		Let $\alpha\geq 0$ and $\bar{f}\in C(-1,1)$. We define $T_0\bar{f}:=\bar{f}$ and for $\alpha >0$,
		\begin{equation*}
			T_{\alpha} \bar{f}(s)
			:=  (1-s^2)^{\frac{\alpha}{2}}\bar{f}(s) + \alpha s\int_{0}^s \bar{f}(t) (1-t^2)^{\frac{\alpha-2}{2}}dt,
			\qquad s\in(-1,1).
		\end{equation*}
	\end{defi}
	
	Note that integrating on $(0,s)$ instead of $(-1,s)$ alters the outcome only by a linear function, however this domain of integration turns out to be convenient in later computations. We will prove that if $\varphi_{i,f}=\psi_{i,g}$, then $\bar{g}$ must be related to $\bar{f}$ (up to the addition of linear functions) via the transform $T_{n-i-1}$.

	Next, as we have defined the transform $T_{n-i-1}$ and computed the $\BB$-mixed and $\DD$-mixed spherical projections of zonal functions, we can show that the diagram in \cref{fig:commuting_diagram} commutes. This will ensure that whenever $\bar{g}=T_{n-i-1}\bar{f}$, then the valuations $\varphi_{i,f}$ and $\psi_{i,g}$ agree on subspaces $E\in\Gr_{i+1}(\R^n)$ containing $e_n$. We require the following technical lemma.	
	
	\begin{lem}
		For all $\alpha>0$ and $x,t\in (-1,1)$,
		\begin{equation}\label{eq:technical_integral}
			\int_x^t s(1-s^2)^{-\frac{\alpha+2}{2}}\abs{s^2-t^2}^{\frac{\alpha-2}{2}}ds
			= \frac{(1-x^2)^{-\frac{\alpha}{2}}\abs{t^2-x^2}^{\frac{\alpha}{2}}}{\alpha(1-t^2)}.
		\end{equation}
	\end{lem}
	\begin{proof}
		Fix the parameters $\alpha>0$ and $t\in (-1,1)$ and observe that the right hand side of \cref{eq:technical_integral} defines a continuous function of $x\in (-1,1)$ that vanishes at $x=t$ and is differentiable on $(-1,1)\setminus\{t\}$. Differentiating the right hand side at $x\in(-1,1)\setminus\{t\}$ yields		
		\begin{align*}
			\frac{d}{dx} \frac{(1-x^2)^{-\frac{\alpha}{2}}\abs{t^2-x^2}^{\frac{\alpha}{2}}}{\alpha(1-t^2)}
			= - x(1-x^2)^{-\frac{\alpha+2}{2}}\abs{x^2-t^2}^{\frac{\alpha-2}{2}}.
		\end{align*}
		Hence, by the fundamental theorem of calculus, we obtain \cref{eq:technical_integral}.
	\end{proof}
	
	\begin{lem}\label{lem:diagram_commutes}
		Let $\alpha>0$ and $\bar{f}\in C(-1,1)$. Then $\pi_{\alpha,\DD}T_{\alpha}\bar{f} = \pi_{\alpha,\BB}\bar{f}$.
	\end{lem}
	\begin{proof}
		Define a function $\bar{g}\in C(-1,1)$ by $\bar{g}:=T_{\alpha}\bar{f}$, that is,
		\begin{equation*}
			\bar{g}(s)
			= (1-s^2)^{\frac{\alpha}{2}}\bar{f}(s) + \alpha s\int_{0}^s \bar{f}(x) (1-x^2)^{\frac{\alpha-2}{2}} dx,
			\qquad s\in (-1,1).
		\end{equation*}
		By a change of variables, we have that
		\begin{align*}
			\pi_{\alpha,\DD}\bar{g}(t)
			&= \omega_{\alpha}(1-t^2) \int_0^1 \bar{g}(st)(1-s^2t^2)^{-\frac{\alpha+2}{2}} (1-s^2)^{\frac{\alpha-2}{2}} ds \\
			&=  \frac{\omega_{\alpha}}{t^{\alpha-1}}(1-t^2)\int_0^t \bar{g}(s) (1-s^2)^{-\frac{\alpha+2}{2}} (t^2-s^2)^{\frac{\alpha-2}{2}} ds.
		\end{align*}
		Next, inserting one integral expression into the other and changing the order of integration yields
		\begin{align*}
			&\int_0^t s\int_{0}^s \bar{f}(x)(1-x^2)^{\frac{\alpha-2}{2}}dx~ (1-s^2)^{-\frac{\alpha+2}{2}} (t^2-s^2)^{\frac{\alpha-2}{2}} ds \\
			&\qquad = \int_0^t \bar{f}(x)(1-x^2)^{\frac{\alpha-2}{2}}\int_{x}^t s(1-s^2)^{-\frac{\alpha+2}{2}} (t^2-s^2)^{\frac{\alpha-2}{2}} ds~  dx \\
			&\qquad = \frac{1}{\alpha(1-t^2)} \int_0^t \bar{f}(x)\frac{1}{1-x^2}(t^2-x^2)^{\frac{\alpha}{2}}dx,
		\end{align*}
		where the final equality is due to \cref{eq:technical_integral}. Consequently, we obtain that
		\begin{align*}
			&\pi_{\alpha,\DD}\bar{g}(t)
			= \frac{\omega_{\alpha} }{t^{\alpha-1}} \int_0^t \bar{f}(x)\frac{1}{1-x^2}\left((t^2-x^2)^{\frac{\alpha}{2}} + (1-t^2)(t^2-x^2)^{\frac{\alpha-2}{2}} \right) dx \\
			&\qquad = \frac{\omega_{\alpha}}{t^{\alpha-1}} \int_0^t \bar{f}(x) (t^2-x^2)^{\frac{\alpha-2}{2}} dx
			%= \omega_{\alpha} \int_0^1 \bar{f}(xt)(1-x^2)^{\frac{\alpha-2}{2}}dx
			=  {\pi}_{\alpha,\BB} \bar{f}(t),
		\end{align*}
		where the final equality is again due to a change of variables.	
	\end{proof}
	
	The uniqueness of the respective integral kernels in \cref{thm:zonalContBallHadwiger} and \cref{thm:zonalDiskHadwiger} will be deduced from the following.
	
	\begin{prop}\label{lem:pi_alpha_injective}
		For $\alpha>0$, the maps $\pi_{\alpha,\BB}$ and $\pi_{\alpha,\DD}$ are injective and map linear functions to linear functions.
	\end{prop}
	\begin{proof}
		First, observe that the map $\pi_{\alpha,\BB}$ is an instance of the $R_{a,b}$ transform defined in \cref{app:Rabtrans}, which are all injective by \cref{lem:Rabinj}.
		Similarly, the map $\pi_{\alpha,\DD}$, as a composition of an $R_{a,b}$ transform and two maps of the form $\bar{f}\mapsto (1-t^2)^{\beta}\bar{f}(t)$, is injective.
		
		Clearly, $\pi_{\alpha,\BB}$ maps linear functions to linear functions.
		A direct computation shows that $T_\alpha$ maps the function $\bar{f}(t):=t$ to itself, so \cref{lem:diagram_commutes} yields that $\pi_{\alpha,\DD}$ also maps linear functions to linear functions.
	\end{proof}
	
	\begin{lem}\label{lem:uniqueness_f_g}
		Let $1\leq i\leq n-1$ and  $f,g \in C(\S^{n-1})$ be zonal.
		\begin{enumerate}[label=\upshape(\roman*)]
			\item \label{lem:uniqueness_f}
			If $\varphi_{i,f}=0$, then $f$ is a zonal linear function.
			\item \label{lem:uniqueness_g}
			If $\psi_{i,g}=0$, then $g$ is a zonal linear function.
		\end{enumerate}
	\end{lem}
	\begin{proof}
		Statement~\ref{lem:uniqueness_f} follows immediately from \cref{thm:mixed_area_meas_liftSmooth}, \cref{sph_proj_zonal}, the uniqueness result in \cref{thm:McMullen_Val_n-1}, and \cref{lem:pi_alpha_injective}.
		Similarly, statement~\ref{lem:uniqueness_g} follows immediately from \cref{thm:liftAreaMeasDisk}, \cref{mixed_sph_proj_zonal_disk}, the uniqueness result in \cref{thm:McMullen_Val_n-1}, and \cref{lem:pi_alpha_injective}.
	\end{proof}
	
	%%%%%%%%%%%%%%%%%%%%%%%%%%%%%%%%%%%%%%%%%%%%%%%%%%%%%%%%%%%%%%%%%%%%%%%%%%%
	\subsection{Extending from subspaces}
	\label{sec:ExtensionThm}
	%%%%%%%%%%%%%%%%%%%%%%%%%%%%%%%%%%%%%%%%%%%%%%%%%%%%%%%%%%%%%%%%%%%%%%%%%%%

	Next, for $E\in\Gr_{i+1}(\R^n)$ containing $e_n$, we show that a zonal valuation $\varphi_{i,f_{\! E}}^{E}$ on $E$ always extends to a zonal valuation $\varphi_{i,f}$ on $\R^n$, provided that $f_{\! E}$ is smooth; this is the content of \cref{thm:ExtendingSmooth}. For the proof, we need the following basic lemma. We denote by $C^\infty[-1,1]$ the space of $C[-1,1]$ functions that are infinitely differentiable on $(-1,1)$ and also posses all (one-sided) higher order derivatives at $\pm 1$.
	
	\begin{lem}\label{lem:smooth_fct_sphere_interval}
		Let $f=\bar{f}(\pair{e_n}{{}\cdot{}}):\S^{n-1}\to\R$ be zonal. Then $f\in C^\infty(\S^{n-1})$ if and only if $\bar{f}\in C^\infty[-1,1]$.
	\end{lem}
	\begin{proof}
		We can parametrize the unit sphere by $u = (\cos\theta)e_n + (\sin\theta)v$ with $\theta\in\R$ and $v\in\S^{n-2}(e_n^{\perp})$. Then $f(u)=\bar{f}(\cos\theta)$, which shows that $f\in C^{\infty}(\S^{n-1})$ if and only if $\bar{f}(\cos\theta)$ is a smooth function of $\theta$.
		If $\bar{f}\in C^\infty[-1,1]$, then $\bar{f}(\cos\theta)$ is a smooth function of $\theta$ by the chain rule.
		
		Conversely, suppose that $\bar{f}(\cos\theta)$ is a smooth function of $\theta$. Then clearly $\bar{f}\in C^\infty(-1,1)$, and it remains to show the existence of all higher order derivatives at $\pm 1$. Since $\bar{f}(\cos\theta)$ is an even, smooth function, there is a smooth function $\tilde{f}$ such that $\bar{f}(\cos\theta)=\tilde{f}(\theta^2)$. Similarly, there is a smooth function $\tilde{q}$ such that $\cos\theta=\tilde{q}(\theta^2)$. By L'Hôpital's rule,
		\begin{equation*}
			\tilde{q}'(0)
			= \lim_{\theta\to 0} \frac{\tilde{q}(\theta^2) - \tilde{q}(0)}{\theta^2}
			= \lim_{\theta\to 0} \frac{\cos\theta - 1}{\theta^2}
			= - \frac{1}{2}
			\neq 0,
		\end{equation*}
		so there exists some neighborhood of zero where $\tilde{q}$ is invertible and its inverse is also smooth. Hence, if $t$ is close to $1$, then $\bar{f}(t)=\tilde{f}((\arccos t)^2)=\tilde{f}(\tilde{q}^{-1}(t))$, so by the chain rule $\bar{f}\in C^\infty(-1,1]$. The argument for $\bar{f}\in C^\infty[-1,1)$ is analogous.
	\end{proof}
	
	\cref{thm:ExtendingSmooth} is now an easy consequence of what we have shown so far and our study of integral transforms in \cref{app:Rabtrans}.	

	\begin{proof}[Proof of \cref{thm:ExtendingSmooth}]
		Due to \cref{thm:mixed_area_meas_liftSmooth}, proving the theorem corresponds to finding a zonal function $f \in C^\infty(\S^{n-1})$ such that
		\begin{align*}
			f_{\! E}
			= \frac{1}{\binom{n-1}{i}}\pi_{\! E,\BB} f.
		\end{align*}
		Writing $f=\bar{f}(\pair{e_n}{{}\cdot{}})$ and $f_{\! E}=\bar{f}_{\! E}(\pair{e_n}{{}\cdot{}})$, by \cref{sph_proj_zonal,lem:smooth_fct_sphere_interval}, this is equivalent to finding a function $\bar{f}\in C^\infty[-1,1]$ such that
		\begin{align*}
			\bar{f}_{\! E}
			= \frac{1}{\binom{n-1}{i}} \overline{\pi}_{n-i-1,\BB} \bar{f}
			= \frac{\omega_{n-i-1}}{\binom{n-1}{i}} R_{1,\frac{n-i-1}{2}} \bar{f},
		\end{align*}
		where $R_{a,b}$ is the transform defined in \cref{app:Rabtrans}. According to \cref{lem:Rabsurj}, such a function $\bar{f}_{\! E}\in C^\infty[-1,1]$ exists, concluding the argument.
	\end{proof}
	
	%%%%%%%%%%%%%%%%%%%%%%%%%%%%%%%%%%%%%%%%%%%%%%%%%%%%%%%%%%%%%%%%%%%%%%%%%%%
	\section{The Klain--Schneider theorem for zonal valuations}
	\label{sec:zonalKlainSchneider}
	%%%%%%%%%%%%%%%%%%%%%%%%%%%%%%%%%%%%%%%%%%%%%%%%%%%%%%%%%%%%%%%%%%%%%%%%%%%
	
	In this section we establish \cref{thm:zonalKlainSchneider}, the zonal analogue of the classical Klain--Schneider theorem and centerpiece of the Klain approach. The main step in the proof is to eliminate the $(n-2)$-homogeneous component, which is subsumed in the following theorem.
	
	\begin{thm} \label{thm:zonalKlainSchneider_i=n-2}
		Let $\varphi\in\Val_{n-2}(\R^n)$ be zonal. If $\varphi$ vanishes on some hyperplane $H\in\Gr_{n-1}(\R^n)$ such that $e_n\in H$, then $\varphi=0$.
	\end{thm}
	
	We will prove this theorem by induction on the dimension $n\geq 3$; the three-dimensional case will be the induction base.

	%%%%%%%%%%%%%%%%%%%%%%%%%%%%%%%%%%%%%%%%%%%%%%%%%%%%%%%%%%%%%%%%%%%%%%%%%%%
	\subsection{The three-dimensional case}
	\label{sec:zonalKlainSchneider:n=3}
	%%%%%%%%%%%%%%%%%%%%%%%%%%%%%%%%%%%%%%%%%%%%%%%%%%%%%%%%%%%%%%%%%%%%%%%%%%%
	
	First, we consider \cref{thm:zonalKlainSchneider_i=n-2} in three dimensions. To this end, we prove the one-homogeneous instance of \cref{thm:zonalDiskHadwiger}, using our computation of the area measures of cones.
	
	\begin{prop}\label{Val1_zonalDiskHadwiger}
		For every zonal valuation $\varphi \in \Val_1(\R^n)$, there exists a zonal function $g\in C(\S^{n-1})$ such that
		\begin{equation}\label{eq:Val1_zonalDiskHadwiger}
			\varphi(K)
			= \int_{\S^{n-1}} g(u) \, dS_1(K,\DD,u),
			\qquad K\in\K(\R^n).
		\end{equation}
	\end{prop}
	\begin{proof}
		Take $g$ to be $g=\bar{g}_{\varphi}(\pair{e_n}{{}\cdot{}})$, where $\bar{g}_{\varphi}$ is defined as in \cref{eq:g_extracted}.
		Due to \cref{lem:sphiC_s}, the function $g$ is continuous on $\S^{n-1}$, and by \cref{eq:area_meas_cone_disk}, the valuations $\varphi$ and $\psi_{i,g}$ coincide on the family of cones $C_s$ for $s\in[-1,1]\setminus\{0\}$.
		
		Next, observe that the valuation property implies that $\varphi$ and $\psi_{i,g}$ coincide on truncated cones and subsequently, on all bodies of revolution with axis $e_n$ that have a polytopal cross-section by two-dimensional planes containing $e_n$. By continuity, $\varphi$ and $\psi_{i,g}$ agree on all bodies of revolution with axis $e_n$.
		
		For a general body $K \in \K(\R^n)$, we define a body of revolution $\overline{K}\in\K(\R^n)$ by
		\begin{equation*}
			h_{\overline{K}}(x)
			= \int_{\SO(n-1)} h_{K}(\vartheta^{-1}x)~d\vartheta
			= \int_{\SO(n-1)} h_{\vartheta K}(x)~d\vartheta,
			\qquad x\in\R^n,
		\end{equation*}
		where integration is with respect to the unique invariant probability measure on $\SO(n-1)$. Hence, by the invariance, Minkowski additivity, and continuity of the valuations $\varphi$ and $\psi_{i,g}$,
		\begin{equation*}
			\varphi(K)
			= \int_{\SO(n-1)} \varphi(\vartheta K)~d\vartheta
			= \varphi(\overline{K})
			= \psi_{i,g}(\overline{K})
			= \int_{\SO(n-1)} \psi_{i,g}(\vartheta K)~d\vartheta
			= \psi_{i,g}(K).
		\end{equation*}
		This shows that $\varphi=\psi_{i,g}$, which concludes the argument.
	\end{proof}
		
	\begin{lem}\label{lem:indStartn3ZonSimple}
		Let $\varphi\in\Val_1(\R^3)$ be zonal. If $\varphi$ vanishes on some plane \linebreak $E\in\Gr_2(\R^3)$  such that $e_3\in E$, then $\varphi=0$.
	\end{lem}
	\begin{proof}
		By \cref{Val1_zonalDiskHadwiger}, $\varphi$ admits an integral representation \cref{eq:Val1_zonalDiskHadwiger} with some zonal function $g=\bar{g}(\pair{e_3}{{}\cdot{}})\in C(\S^{2})$.
		Suppose now that $\varphi$ vanishes on a plane $E\in\Gr_2(\R^3)$ containing $e_3$. Then \cref{thm:liftAreaMeasDisk} and \cref{mixed_sph_proj_zonal_disk} imply that for all $K\in\K(E)$,
		\begin{equation*}
			\int_{\S^{1}(E)} (\pi_{1,\DD}\bar{g})(\pair{e_3}{u})\, dS_{1}^E(K,u)
			= 2 \varphi(K)
			= 0,
		\end{equation*}
		and thus, ${\pi}_{1,\DD}\bar{g}$ is a linear function. Hence, by \cref{lem:pi_alpha_injective}, $g$ is a linear function, and subsequently $\varphi=0$.
	\end{proof}
	
	%%%%%%%%%%%%%%%%%%%%%%%%%%%%%%%%%%%%%%%%%%%%%%%%%%%%%%%%%%%%%%%%%%%%%%%%%%%
	\subsection{The induction step}
	\label{sec:zonalKlainSchneider:n>3}
	%%%%%%%%%%%%%%%%%%%%%%%%%%%%%%%%%%%%%%%%%%%%%%%%%%%%%%%%%%%%%%%%%%%%%%%%%%%
	
	Now we pass from three dimensions to general dimensions. One of the main ideas is to show that the valuation in question vanishes on certain orthogonal sums of convex bodies. First, we need the following easy lemma. Recall that we globally assumed the dimension to be $n\geq 3$.
	
	\begin{lem}\label{lem:glueingLinearMaps}
		Let $\ell: \R^n \to \R$ be such that its restriction to $e_n^\perp$ and to each hyperplane containing $e_n$ is a linear function. Then $\ell$ is a linear function.
	\end{lem}
	\begin{proof}
		By assumption, we find $x_0 \in e_n^\perp$ such that $\ell|_{e_n^\perp} = \pair{x_0}{{}\cdot{}}$ and for every hyperplane $H$ containing $e_n$, we find $x_H\in H$ such that $\ell|_H = \pair{x_H}{{}\cdot{}}$.
		As $n \geq 3$, we have $e_n^\perp \cap H \neq \{o\}$ and we can consider $\ell|_{e_n^\perp \cap H}$ to deduce that
		\begin{align*}
			P_H x_0 = P_{H\cap e_n^\perp} x_0 = P_{H\cap e_n^\perp} x_H = P_{e_n^\perp} x_H,
		\end{align*}
		where $P_E$ denotes the orthogonal projection onto a subspace $E\subseteq \R^n$ and we used the fact that $P_{H\cap e_n^\perp}=P_HP_{e_n^\perp}=P_{e_n^\perp}P_H$. Next, by plugging $e_n$ into $\ell$, we see that $\ell(e_n) = \pair{x_H}{e_n}$ for each hyperplane $H$ containing $e_n$. Consequently,
		\begin{align*}
			x_H = \pair{x_H}{e_n}e_n + P_{e_n^\perp}x_H = \ell(e_n) e_n + P_H x_0 = P_H ( \ell(e_n) e_n + x_0),
		\end{align*}
		and we conclude that the linear function $\pair{\ell(e_n) e_n + x_0}{{}\cdot{}}$ coincides with $\ell$ on every hyperplane $H$ containing $e_n$, and thus, everywhere.
	\end{proof}
	
	For the next lemma, we require the following classical result of McMullen.
	
	\begin{thm}[{\cite{McMullen1980}}]\label{thm:McMullen_Val_n-1}
		For every valuation $\varphi\in\Val_{n-1}(\R^n)$, there exists a function $f\in C(\S^{n-1})$ such that
		\begin{equation*}
			\varphi(K)
			= \int_{\S^{n-1}} f(u)\, dS_{n-1}(K,u),
			\qquad K\in\K(\R^n).
		\end{equation*}
		Moreover, $f$ is unique up to the addition of a linear function.
	\end{thm}
	
	\begin{lem}\label{lem:Codeg1ValVanOnCylVan}
		Let $\varphi \in \Val_{n-1}(\R^n)$ and suppose that 
		\begin{align*}
			\varphi(K+I) = 0 \qquad \text{for all } K \in \K(H) \text{ and } I \in \K(H^\perp)
		\end{align*}
		whenever $H=e_n^\perp$ or $H$ is a hyperplane containing $e_n$. Then $\varphi = 0$.
	\end{lem}
	\begin{proof}
		By \cref{thm:McMullen_Val_n-1}, there exists a function $f \in C(\S^{n-1})$ such that
		\begin{align*}
			\varphi(K) = \int_{\S^{n-1}} f(u) \, dS_{n-1}(K,u), \qquad K \in \K(\R^n).
		\end{align*}
		Let now $H$ be some hyperplane such that $\varphi(K+I)=0$ for all $K\in\K(H)$ and $I\in\K(H^\perp)$. Since $K+I$ is a cylinder over $K$, its boundary splits naturally, and thus, so does its surface area measure, yielding
		\begin{align*}
			&0=\varphi(K+I)
			= \int_{\S^{n-1}} f(u)\, dS_{n-1}(K+I,u) \\
			&\qquad = V_{n-1}(K) (f(w) + f(-w)) + V_1(I) \int_{\S^{n-2}(H)} f(u)\, dS_{n-2}^H (K,u),
		\end{align*}
		where $w\in\S^{n-1}$ is such that $H=w^\perp$. By choosing $I=\{o\}$, we see that $f(w)+f(-w)=0$. By choosing $I\neq\{o\}$, we see that the final integral expression vanishes for all $K\in\K(H)$, and thus, the restriction $f|_{\S^{n-2}(H)}$ is linear.
		
		Finally, let $\ell:\R^n\to\R$ denote the one-homogeneous extension of $f$ to $\R^n$, that is, $\ell(o)=0$ and $\ell(x)=\norm{x}f(x/\norm{x})$ for $x\neq o$. \cref{lem:glueingLinearMaps} shows that $\ell$ is a linear function on $\R^n$, so $f$ is a linear function on $\S^{n-1}$, and thus, $\varphi = 0$.
	\end{proof}
	
	We are now ready to prove \cref{thm:zonalKlainSchneider_i=n-2}.
	
%	\begin{thm}
%		Let $\varphi\in\Val_{n-2}(\R^n)$ be zonal. If $\varphi$ vanishes on some hyperplane $H\in\Gr_{n-1}(\R^n)$ such that $e_n\in H$, then $\varphi=0$.
%	\end{thm}

	\begin{proof}[Proof of \cref{thm:zonalKlainSchneider_i=n-2}]
		We prove the theorem by induction on the dimension \linebreak $n\geq 3$.
		The three-dimensional case is precisely the content of \cref{lem:indStartn3ZonSimple}.
		
	\begin{figure}[t]
    	\centering
    	\includegraphics[width=0.5\textwidth]{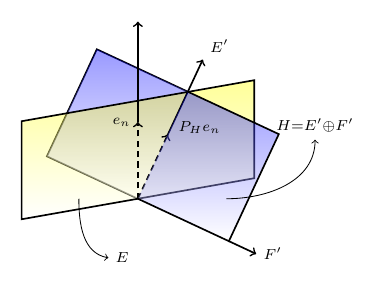}   
    	\vspace{-20pt}	
    	\caption{We extend the orthogonal sum $H=E'\oplus F'$, where $P_H e_n\in E'$, to an orthogonal sum $\R^n=E\oplus F'$.} 
    	\label{fig:Klain_step} 	 
	\end{figure}		
		
		For the induction step, let $n>3$ and take some zonal $\varphi\in\Val_{n-2}(\R^n)$ that vanishes on some, and thus, on every hyperplane $H\in\Gr_{n-1}(\R^n)$ containing~$e_n$. Consider a proper orthogonal sum $\R^n=E\oplus F$, where $e_n\in E$. We claim that $\varphi(K+L)=0$ for all $K\in\K(E)$ and $L\in\K(F)$. To show this, observe that for fixed $K$, the map $\varphi(K + {}\cdot{})$ defines a continuous and rigid motion invariant valuation on $F$. According to \cref{thm:Hadwiger_intrinsic_vols}, it must be a linear combination of intrinsic volumes. Since $\varphi(K + {}\cdot{})$ is also a simple valuation on $F$, it is a multiple of $\mathrm{vol}_{F}$, the only simple intrinsic volume on $F$. That is, there exists some map $\psi_E:\K(E)\to\R$ such that $\varphi(K+L)=\psi_E(K)\mathrm{vol}_{F}(L)$ for all $K\in\K(E)$ and $L\in\K(F)$. Fixing the body $L$ reveals that $\psi_E$ is a continuous, translation invariant, and zonal valuation on $E$. Moreover, $\psi_E$ is homogeneous of degree $\dim E-2$ and vanishes on all hyperplanes of $E$ containing $e_n$. By induction hypothesis, $\psi_E=0$, and thus, $\varphi(K+L)=0$ for all $K\in\K(E)$ and $L\in\K(F)$.
		
		Next, we want to show that $\varphi$ vanishes on every hyperplane $H \in \Gr_{n-1}(\R^n)$. If $e_n\in H$, then this is due to our assumption on $\varphi$; for $H = e_n^\perp$, this is due to the previous step. Otherwise, consider a proper orthogonal sum $H=E'\oplus F'$, where $P_H e_n\in E'$ and $P_H$ denotes the orthogonal projection onto $H$. Note that for every $x\in F'$,
		\begin{align*}
			\pair{x}{e_n}
			= \pair{P_H x}{e_n}
			= \pair{x}{P_H e_n}
			= 0,
		\end{align*}
		hence $F' \subseteq e_n^\perp$. Consequently, if we define $E = \mathrm{span}(E'\cup\{e_n\})$ and $F=F'$, we obtain the proper orthogonal sum $\R^n = E \oplus F$ (see \cref{fig:Klain_step}). Since $E'\subseteq E$, the previous step implies that $\varphi(K + L) = 0$ for all $K \in \K(E')$ and $L \in \K(F')$. Therefore, the restriction $\varphi|_H$ meets the requirements of \cref{lem:Codeg1ValVanOnCylVan}, so $\varphi|_H = 0$. This shows that the valuation $\varphi$ is simple, so \cref{thm:klainschneider} implies that $\varphi = 0$, concluding the proof.
	\end{proof}
	
	As we have announced at the beginning of this section, the main part in the proof of \cref{thm:zonalKlainSchneider} is to eliminate the $(n-2)$-homogeneous component. Now that we have dealt with this case, it remains to handle the other homogeneous cases and reduce the general case to these.	
	
	\begin{lem}\label{lem:valuation_components_restricting}
		If a valuation $\varphi \in \Val(\R^n)$ vanishes on a subspace $E\subseteq \R^n$, then so do all of its homogeneous components.
		%and their respective even and odd parts.
	\end{lem}
	\begin{proof}
		Suppose that $\varphi\in\Val(\R^n)$ vanishes on the subspace $E\subseteq \R^n$ and let $\varphi=\varphi_0+\cdots+\varphi_n$ denote its homogeneous decomposition. Then for $K\in\K(E)$ and $\lambda\geq 0$,
		\begin{align*}
			0 = \varphi(\lambda K) = \sum_{i=0}^{n} \varphi_i(\lambda K) = \sum_{i=0}^{n} \lambda^i \varphi_i(K).
		\end{align*}
		By comparison of coefficients, we see that the homogeneous components of $\varphi$ vanish on $E$.
		%The even and odd part of $\varphi_i$, denoted by $\varphi_i^+$ and $\varphi_i^-$, respectively, is given by $\varphi_i^{\pm}(K)=\frac{1}{2}(\varphi_i(K) \pm \varphi_i(-K))$, and thus, must also vanish on $E$.		
	\end{proof}
	
	We can now finally prove \cref{thm:zonalKlainSchneider}.
	
	\begin{proof}[Proof of \cref{thm:zonalKlainSchneider}]
		It is easy to see that every zonal valuation $\varphi\in\Val(\R^n)$ of the form \cref{eq:zonalKlainSchneider} vanishes on all hyperplanes containing $e_n$. Conversely, take a zonal valuation $\varphi\in\Val(\R^n)$ and suppose that it vanishes on some hyperplane $H$ containing $e_n$. Letting $\varphi=\varphi_0+\cdots+\varphi_n$ denote its homogeneous decomposition, each component $\varphi_i\in\Val_i(\R^n)$ is again zonal and vanishes on $H$. In particular, $\varphi_i$ vanishes on all subspaces $E\in\Gr_{n-2}(\R^n)$, so according to \cref{cor:valDetByRestr}, all homogeneous components but $\varphi_{n-2}$, $\varphi_{n-1}$, and $\varphi_n$ must vanish.
		
		By \cref{thm:zonalKlainSchneider_i=n-2}, $\varphi_{n-2}=0$.	Due to \cref{thm:McMullen_Val_n-1} and Hadwiger's classification of $n$-homogeneous valuations~\cite{Hadwiger1957}, the valuation $\varphi$ is thus of the form \cref{eq:zonalKlainSchneider} for some constant $c\in\R$ and function $g\in C(\S^{n-1})$ which must clearly be zonal. In order to see that $g(u)=0$ for $u\in \S^{n-1}\cap e_n^{\perp}$, we evaluate $\varphi$ on the body $K=\BB\cap u^\perp$, where $\BB$ denotes the Euclidean unit ball, which yields
		\begin{equation*}
			0
			= \varphi(K)
			= \int_{\S^{n-1}} g(v) \, dS_{n-1}(K,v)
			= \kappa_{n-1}(g(u)+g(-u)).
		\end{equation*}
		As $g$ is zonal, this shows that $g$ vanishes on $\S^{n-1}\cap e_n^{\perp}$.
	\end{proof}

	Note that the assumption of $\varphi$ being zonal cannot be dropped in \cref{thm:zonalKlainSchneider}. For instance, consider the valuation $\varphi\in\Val_{n-2}(\R^n)$ defined by
	\begin{align*}\label{eq:exZonSimpleCodeg2}
		\varphi(K)= \int_{\S^{n-2}(e_n^\perp)} f(u) \, dS_{n-2}^{e_n^\perp}(K|e_n^\perp,u),
		\qquad K\in\K(\R^n),
	\end{align*}
	for some odd function $f \in C(\S^{n-2}(e_n^\perp))$. Then $\varphi$ vanishes on all hyperplanes containing~$e_n$, but it is not of the form \cref{eq:zonalKlainSchneider}. 
	This raises the question of how to characterize valuations in $\Val(\R^n)$ that vanish on all hyperplanes containing~$e_n$.

	In the case of even valuations, however, it turns out that the assumption of $\SO(n-1)$-invariance can be dropped and the proof is a simple application of the following corollary of \cref{thm:klainschneider}.
	
	\begin{cor}\label{cor:valDetByRestr_even}
		Let $0\leq i\leq n-1$ and $\varphi \in \Val_i(\R^n)$ be even. If $\varphi$ vanishes on all subspaces $E \in \Gr_{i}(\R^n)$, then $\varphi=0$.
	\end{cor}
	
	\begin{prop}\label{zonally_simple_even}
		An even valuation $\varphi\in\Val(\R^n)$ vanishes on all hyperplanes containing $e_n$ if and only if there exist a constant $c\in\R$ and an even function $f\in C(\S^{n-1})$ vanishing on $\S^{n-1}\cap e_n^\perp$ such that
		\begin{equation}\label{eq:zonally_simple_even}
			\varphi(K)
			= c V_n(K) + \int_{\S^{n-1}} f(u) \,dS_{n-1}(K,u),
			\qquad K \in \K(\R^n).
		\end{equation}
	\end{prop}
	\begin{proof}
		As every subspace of dimension less or equal $n-2$ is contained in a hyperplane containing $e_n$, by \cref{cor:valDetByRestr_even}, the proof reduces to considering valuations of degree $n$ and $n-1$. The claim then follows directly from Hadwiger's characterization of $n$-homogeneous valuations \cite{Hadwiger1957} and \cref{thm:McMullen_Val_n-1}.
	\end{proof}
	
	\begin{rem}\label{rem:ksForEvenWoZonal}
		In the Klain--Schneider type theorem in the functional setting, \cite{Colesanti2023a}*{Theorem~1.2}, the valuations are merely assumed to be epi-translation invariant. It might seem that this is somehow more general than \cref{thm:zonalKlainSchneider} since there is no additional rotational invariance imposed. However, it turns out that from \cref{zonally_simple_even} and the correspondence between the functional and the even geometrical setting (see \cite{Knoerr2021}*{Section~3}), one can deduce \cite{Colesanti2023a}*{Theorem~1.2}. This will be discussed in more detail in future work.		
	\end{rem}

	Like the classical Klain--Schneider theorem, \cref{thm:zonalKlainSchneider} entails that zonal valuations are determined by certain restrictions; this is the content of \cref{cor:zonalValDetByRestr}.
	
%	\begin{cor}
%		Let $0\leq i< n-1$ and $\varphi\in\Val_i(\R^n)$ be zonal. If $\varphi$ vanishes on some subspace $E\in\Gr_{i+1}(\R^n)$ such that $e_n\in E$, then $\varphi= 0$.
%	\end{cor}
	\begin{proof}[Proof of \cref{cor:zonalValDetByRestr}]
		For degree $i=n-1$, the statement is trivial. For degrees $1\leq i<n-1$, we prove the claim by induction on the dimension $n\geq 3$. For dimension $n=3$ and degree $i=1$, this is precisely the content of \cref{lem:indStartn3ZonSimple}.

		For the induction step, let $n>3$ and $1\leq i < n-1$. Take a zonal valuation $\varphi\in\Val_i(\R^n)$ that vanishes on some subspace $E\in\Gr_{i+1}(\R^n)$ containing~$e_n$. Choose a hyperplane $H\in\Gr_{n-1}(\R^n)$ such that $H\supseteq E$. Then $\varphi|_H$ is a zonal valuation on $H$ and by the induction hypothesis, $\varphi|_H=0$. Consequently, $\varphi$ meets the requirements of \cref{thm:zonalKlainSchneider}, so due to its homogeneity, $\varphi=0$.
	\end{proof}
	
	As a consequence of \cref{cor:zonalValDetByRestr} and the commuting diagram in \cref{fig:commuting_diagram}, we obtain that the transform $T_{n-i-1}$ allows us to move between integral representations as expected.
	
	\begin{cor}\label{cor:integral_reps_coincide}
		Let $1\leq i\leq n-1$ and $f=\bar{f}(\pair{e_n}{{}\cdot{}}),g=\bar{g}(\pair{e_n}{{}\cdot{}}) \in C(\S^{n-1})$. If $\bar{g}=T_{n-i-1}\bar{f}$, then $\varphi_{i,f}=\psi_{i,g}$.
	\end{cor}
	\begin{proof}
		Consider a subspace $E\in\Gr_{i+1}(\R^n)$ containing $e_n$. By \cref{thm:mixed_area_meas_liftSmooth} and \cref{sph_proj_zonal},
		\begin{align*}
			\varphi_{i,f}(K)
			= \frac{1}{\binom{n-1}{i}} \int_{\S^{i}(E)} (\pi_{n-i-1,\BB}\bar{f})(\pair{e_n}{u}) \, dS_i^E(K,u),
			\qquad K\in\K(E).
		\end{align*}	
		Similarly, by \cref{thm:liftAreaMeasDisk} and \cref{mixed_sph_proj_zonal_disk},
		\begin{equation*}
			\psi_{i,g}(K)
			= \frac{1}{\binom{n-1}{i}} \int_{\S^{i}(E)} (\pi_{n-i-1,\DD}\bar{g})(\pair{e_n}{u}) \, dS_i^E(K,u),
			\qquad K\in\K(E).
		\end{equation*}	
		Since $\bar{g}=T_{n-i-1}\bar{f}$, \cref{lem:diagram_commutes} yields $\pi_{n-i-1,\BB}\bar{f}=\pi_{n-i-1,\DD}\bar{g}$, so the valuations $\varphi_{i,f}$ and $\psi_{i,g}$ coincide on $E$. Hence, \cref{cor:zonalValDetByRestr} implies that $\varphi_{i,f}=\psi_{i,g}$.
	\end{proof}

	%%%%%%%%%%%%%%%%%%%%%%%%%%%%%%%%%%%%%%%%%%%%%%%%%%%%%%%%%%%%%%%%%%%%%%%%%%%
	\section{Hadwiger type theorems for zonal valuations}
	\label{sec:zonalHadwigerThms}
	%%%%%%%%%%%%%%%%%%%%%%%%%%%%%%%%%%%%%%%%%%%%%%%%%%%%%%%%%%%%%%%%%%%%%%%%%%%
	
	In this section, we establish several integral representations for zonal valuations using the Klain approach. First, we recover a Hadwiger type theorem for smooth, zonal valuations by Schuster and Wannerer~\cite{Schuster2018}. From this we deduce \cref{thm:zonalDiskHadwiger}, and, finally, we also obtain \cref{thm:zonalContBallHadwiger}.
	
	%%%%%%%%%%%%%%%%%%%%%%%%%%%%%%%%%%%%%%%%%%%%%%%%%%%%%%%%%%%%%%%%%%%%%%%%%%%
	\subsection{Smooth Valuations}
	\label{sec:SmoothHadwiger}
	%%%%%%%%%%%%%%%%%%%%%%%%%%%%%%%%%%%%%%%%%%%%%%%%%%%%%%%%%%%%%%%%%%%%%%%%%%%
	
	Recall that the space $\Val(\R^n)$ is a Banach space, when endowed with the norm $\norm{\varphi}=\sup\{\abs{\varphi(K)}:K\subseteq\BB\}$. Moreover, there is a natural representation of the group $\mathrm{GL}(n)$ on this space: For $\varphi\in\Val(\R^n)$ and $\vartheta\in\mathrm{GL}(n)$, we set
	\begin{equation*}
		(\vartheta\cdot\varphi)(K)
		= \varphi(\vartheta^{-1}(K)),
		\qquad K\in\K(\R^n).
	\end{equation*}
	A valuation $\varphi\in\Val(\R^n)$ is called \emph{smooth} if the map $\mathrm{GL}(n) \to \Val(\R^n): \vartheta \mapsto \vartheta\cdot\varphi$ is infinitely differentiable.
	
	For $(n-1)$-homogeneous smooth valuations, we have the following integral representation. It is a corollary of the classical \cref{thm:McMullen_Val_n-1} by McMullen~\cite{McMullen1980}.
	
	\begin{cor}[{\cite{McMullen1980}}]\label{thm:McMullen_Val_n-1_smooth}
		For every smooth valuation $\varphi\in\Val_{n-1}(\R^n)$, there exists a function $f\in C^\infty(\S^{n-1})$ such that
		\begin{equation*}
			\varphi(K)
			= \int_{\S^{n-1}} f(u)\, dS_{n-1}(K,u),
			\qquad K\in\K(\R^n).
		\end{equation*}
		Moreover, $f$ is unique up to the addition of a linear function.
	\end{cor}
	
	The fact that $f\in C^\infty(\S^{n-1})$ whenever $\varphi$ is smooth can be easily obtained in a similar fashion as in the proof of McMullen for the continuity of~$f$. By combining this with \cref{cor:zonalValDetByRestr,thm:ExtendingSmooth}, we recover the following Hadwiger type theorem about smooth, zonal valuations by Schuster and Wannerer~\cite{Schuster2018}.	
	
	\begin{thm}[{\cite{Schuster2018}}]\label{thm:zonalSmoothHadwigerHomogeneous}
		Let $1\leq i\leq n-1$. Then for every smooth, zonal valuation $\varphi\in\Val_i(\R^n)$, there exists a zonal function $f\in C^\infty(\S^{n-1})$ such that
		\begin{equation}\label{eq:zonalSmoothHadwigerHomogeneous}
			\varphi(K)
			= \int_{\S^{n-1}} f(u) \, dS_i(K,u),
			\qquad K\in\K(\R^n).
		\end{equation}
		Moreover, $f$ is unique up to the addition of a zonal linear function.
	\end{thm}
	\begin{proof}
		The uniqueness of $f$ follows from \cref{lem:uniqueness_f_g}~\ref{lem:uniqueness_f}.
		
		For $i=n-1$, we apply \cref{thm:McMullen_Val_n-1_smooth} to obtain integral representation~\cref{eq:zonalSmoothHadwigerHomogeneous} with a function $f\in C^\infty(\S^{n-1})$. Since $\varphi$ is zonal, so is $f$.
		
		For $i<n-1$, choose some subspace $E\in\Gr_{i+1}(\R^n)$ such that $e_n\in E$ and consider the restriction $\varphi|_E \in \Val_i(E)$. Then $\varphi|_E$, as a valuation on $E$, is smooth (with respect to the natural representation of $\mathrm{GL}(E)$) and zonal. Therefore, by the first part of the proof, there exists a zonal function $f_{\! E}\in C^\infty(\S^i(E))$ such that
		\begin{equation*}
			\varphi(K)
			= \int_{\S^i(E)} f_{\! E}(v) \, dS_i^E(K,v),
			\qquad K\in\K(E).
		\end{equation*}
		By \cref{thm:ExtendingSmooth}, there exists a zonal function $f\in C^{\infty}(\S^{n-1})$ such that
		\begin{equation*}
			\varphi(K)
			= \int_{\S^{n-1}} f(u)\, dS_i(K,du),
			\qquad K\in\K(E).
		\end{equation*}
		Observe now that the right hand side defines a valuation $\tilde{\varphi}\in\Val_i(\R^n)$ which agrees with $\varphi$ on $E$. According to \cref{cor:zonalValDetByRestr}, this already implies that $\varphi=\tilde{\varphi}$, yielding the desired integral representation.
	\end{proof}
	
	%%%%%%%%%%%%%%%%%%%%%%%%%%%%%%%%%%%%%%%%%%%%%%%%%%%%%%%%%%%%%%%%%%%%%%%%%%%
	\subsection{Continuous valuations}
	\label{sec:ContHadwiger}
	%%%%%%%%%%%%%%%%%%%%%%%%%%%%%%%%%%%%%%%%%%%%%%%%%%%%%%%%%%%%%%%%%%%%%%%%%%%
	
	We now turn to the Hadwiger type theorems for continuous zonal valuations that we have presented in the introduction: \cref{thm:zonalDiskHadwiger} involving mixed area measures with the disk and \cref{thm:zonalContBallHadwiger} involving the classical area measures. First, we obtain \cref{thm:zonalDiskHadwiger} by an approximation argument from \cref{thm:zonalSmoothHadwigerHomogeneous}. This requires the following lemma which can be proved using a standard convolution argument.
	
	\begin{lem}\label{lem:smooth_Vals_dense}
		Let $1\leq i\leq n-1$. Then for every zonal valuation $\varphi\in \Val_i(\R^n)$, there exists a sequence of smooth, zonal valuations in $\Val_i(\R^n)$ converging to $\varphi$ in the Banach norm.
	\end{lem}

	\begin{proof}[Proof of \cref{thm:zonalDiskHadwiger}]
		The uniqueness of $g$ follows from \cref{lem:uniqueness_f_g}~\ref{lem:uniqueness_g}.
		
		For the existence, by \cref{lem:smooth_Vals_dense}, we may choose a family of smooth, zonal valuations $\varphi^{k}\in\Val_i(\R^n)$ converging to $\varphi$ in the Banach norm as $k\to\infty$. By \cref{thm:zonalSmoothHadwigerHomogeneous}, there exist zonal functions $f_{k}=\bar{f}_{k}(\pair{e_n}{{}\cdot{}})\in C^\infty(\S^{n-1})$ such that $\varphi^{k}=\varphi_{i,f_{k}}$.
		If we let $\bar{g}_{k}:=T_{n-i-1}\bar{f}_{k}$, then $g_{k}=\bar{g}_{k}(\pair{e_n}{{}\cdot{}})\in C(\S^{n-1})$ and \cref{cor:integral_reps_coincide} yields $\varphi^{k}=\psi_{i,g_{k}}$. 
		
		Modifying each $g_k$ by a linear function, if necessary, by the uniqueness part above and \cref{lem:sphiC_s}, the functions $\bar{g}_{k}$ form a Cauchy sequence in $C[-1,1]$, so by completeness, they converge uniformly to some function $\bar{g}\in C[-1,1]$ as $k\to\infty$. If we set $g:=\bar{g}(\pair{e_n}{{}\cdot{}})$, then
		\begin{equation*}
			\varphi(K)
			= \lim_{k\to \infty} \varphi^{k}(K)
			= \lim_{k\to \infty} \psi_{i,g_k}(K)
			= \psi_{i,g}(K),
			\qquad K\in\K(\R^n).
		\end{equation*}
		Thus $\varphi=\psi_{i,g}$, which concludes the argument.
	\end{proof}
	
	As was already indicated in the introduction, we obtain the following two corollaries as a direct consequence of \cref{thm:zonalDiskHadwiger}, \cref{lem:sphiC_s}, and \cref{eq:psi_norm_estimate}.
	
	\begin{cor}\label{cor:zonalValDetByCones}
		Let $1\leq i\leq n-1$ and $\varphi\in\Val_i(\R^n)$ be zonal. If $\varphi(C_s)=0$ for all $s\in[-1,1]\setminus\{0\}$, then $\varphi=0$.
	\end{cor}
	
	\begin{cor}\label{cor:zonal_convergence}
		Let $1\leq i\leq n-1$ and $\varphi^k,\varphi \in\Val_i(\R^n)$ be zonal for $k\in\N$, and let $g_k ,g \in C(\S^{n-1})$ be as in \cref{eq:zonalDiskHadwiger}. Then $\varphi^k\to \varphi$ in the Banch norm if and only if there exist constants $a_k\in\R$ such that $g_k + a_k\pair{e_n}{{}\cdot{}} \to g$ uniformly on $\S^{n-1}$.
	\end{cor}

%	\begin{cor}\label{cor:zonalValDetByCones_non-homogeneous}
%		If $\varphi\in\Val(\R^n)$ is zonal and $\varphi(\lambda C_s)=0$ for all $s\in[-1,1]\setminus\{0\}$ and $\lambda\geq 0$, then $\varphi=0$.
%	\end{cor}
	
	Next, we want to recover \cref{thm:zonalContBallHadwiger}. One key step of the proof is to extend the definition of $\varphi_{i,f}$ from continuous $\bar{f}\in C[-1,1]$ to $\bar{f}\in\Dclass^{n-i-1}$ and to extend \cref{cor:integral_reps_coincide} about moving between the different integral representations. For this, we need the following classical result by Firey~\cite{Firey1970a}.
	
	\begin{thm}[{\cite{Firey1970a}}]\label{thm:estimate_area_measures_of_caps}
		Let $1\leq i\leq n-1$. Then there exists a constant $A_{n,i}>0$ such that for all $K\in\K(\R^n)$, $v\in\S^{n-1}$, and $\varepsilon\geq 0$,
		\begin{equation*}
			S_i(K,\{u\in\S^{n-1}:\pair{u}{v}>1-\varepsilon\})
			\leq A_{n,i}(\mathrm{diam}\, K)^{i}\varepsilon^{\frac{n-i-1}{2}}.
		\end{equation*} 
	\end{thm}
	
	\begin{prop}\label{thm:dictBallDisk+}
		Let $1\leq i\leq n-1$, $\bar{f}\in\Dclass^{n-i-1}$, and $f=\bar{f}(\pair{e_n}{{}\cdot{}})$. Then $f\in C(\S^{n-1}\setminus \{\pm e_n\})$ and there exists a zonal valuation $\varphi_{i,f}\in \Val_i(\R^n)$ such that
		\begin{equation*}
			\varphi_{i,f}(K)
			= \lim_{\varepsilon\to 0^+} \int_{\S^{n-1}\setminus U_{\varepsilon}} f(u) \, dS_i(K,u),
			\qquad K\in \K(\R^n),
		\end{equation*}
		where $U_\varepsilon = \{u\in\S^{n-1}: \abs{\pair{e_n}{u}}> 1-\varepsilon\}$.
		
		Moreover, if $\bar{g}=T_{n-i-1}\bar{f}$, then $g=\bar{g}(\pair{e_n}{{}\cdot{}})\in C(\S^{n-1})$ and $\varphi_{i,f}=\psi_{i,g}$.
	\end{prop}
	\begin{proof}
		The idea is to obtain the statement by approximation and \cref{cor:integral_reps_coincide}. To this end, take a family of bump functions $\bar{\eta}_{\varepsilon} \in C[-1,1]$, $\varepsilon>0$, such that
		\begin{equation*}
			\bar{\eta}_{\varepsilon}(s) = 1 \text{ for } \abs{s}\leq 1-\varepsilon,
			\qquad \bar{\eta}_{\varepsilon}(s) = 0 \text{ for } \abs{s}\geq 1-\tfrac{\varepsilon}{2},
			\qquad \text{and} \qquad 0\leq \bar{\eta}_{\varepsilon} \leq 1.
		\end{equation*}
		For $\varepsilon>0$, we define $\bar{f}_{\varepsilon}=\bar{\eta}_{\varepsilon}\bar{f} \in C[-1,1]$. We also define $\bar{g}_{\varepsilon}=T_{n-i-1}\bar{f}_{\varepsilon}$ and $\bar{g}=T_{n-i-1}\bar{f}$, and observe that $\bar{g}_{\varepsilon },\bar{g}\in C[-1,1]$. We denote the respective zonal extensions of these functions to the unit sphere by $\eta_{\varepsilon},f_{\varepsilon},g_{\varepsilon},g \in C(\S^{n-1})$.
		
		According to \cref{cor:integral_reps_coincide}, we have that $\varphi_{i,f_{\varepsilon}}=\psi_{i,g_{\varepsilon}}$ for all $\varepsilon>0$, and due to \cref{lem:T_convergence}, the functions $g_\varepsilon$ converge to $g$ uniformly on $\S^{n-1}$ as $\varepsilon\to 0^+$. Hence, by \cref{cor:zonal_convergence},
		\begin{equation*}
			\lim_{\varepsilon\to 0^+} \varphi_{i,f_{\varepsilon}}(K)
			= \lim_{\varepsilon\to 0^+} \psi_{i,g_{\varepsilon}}(K)
			= \psi_{i,g}(K),
			\qquad K\in\K(\R^n).
		\end{equation*}
		Consequently, it remains to show that for every given $K\in\K(\R^n)$, the principal value integral in the statement of the proposition exists and agrees with the limit of $\varphi_{i,f_{\varepsilon}}(K)$ as $\varepsilon\to 0^+$. To this end, observe that
		\begin{equation*}
			\varphi_{i,f_{\varepsilon}}(K)- \int_{\S^{n-1}\setminus U_{\varepsilon}} f \, dS_i(K,{}\cdot{})
			%= \int_{U_{\varepsilon}} \eta_{\varepsilon} f \, dS_i(K,{}\cdot{})
			= \int_{U_{\varepsilon}^+} \eta_{\varepsilon} f \, dS_i(K,{}\cdot{}) + \int_{U_{\varepsilon}^-} \eta_{\varepsilon} f \, dS_i(K,{}\cdot{}),
		\end{equation*}
		where $U_\varepsilon^\pm = \{u\in\S^{n-1}: \pm\pair{e_n}{u} > 1-\varepsilon\}$ are spherical caps around the individual poles. Since $U_{\varepsilon}^+$ is connected, by the mean value theorem for integrals, there exists some $t\in (1-\varepsilon,1-\frac{\varepsilon}{2})$ such that
		\begin{equation*}
 		\int_{U_{\varepsilon}^+} \eta_{\varepsilon} f \, dS_i(K,{}\cdot{})
			= \bar{f}(t)\int_{U_{\varepsilon}^+} \eta_{\varepsilon} \, dS_i(K,{}\cdot{}).
		\end{equation*}
		Consequently, using \cref{thm:estimate_area_measures_of_caps}, we can estimate that
		\begin{align*}
			&\left| \int_{U_{\varepsilon}^+} \eta_{\varepsilon} f \, dS_i(K,{}\cdot{}) \right|
			= \abs{\bar{f}(t)} \int_{U_{\varepsilon}^+} \eta_{\varepsilon} \, dS_i(K,{}\cdot{})
			\leq \abs{\bar{f}(t)} S_i(K,U_{\varepsilon}^+) \\
			&\qquad \leq A_{n,i}(\mathrm{diam}\, K)^{i} \abs{\bar{f}(t)}\varepsilon^{\frac{n-i-1}{2}}
			\leq 2^{\frac{n-i-1}{2}} A_{n,i} (\mathrm{diam}\, K)^{i} \abs{\bar{f}(t)} (1-t^2)^{\frac{n-i-1}{2}},
			%\leq 2^{\frac{n-i-1}{2}} A_{n,i,K} \sup_{x \geq 1-\varepsilon}\abs{\bar{f}(x)} (1-x^2)^{\frac{n-i-1}{2}}
		\end{align*}
		where the final inequality uses that $\frac{\varepsilon}{2} < 1-t< 1-t^2$. Since $\bar{f}\in\Dclass^{n-i-1}$, the final term tends to zero as $\varepsilon\to 0^+$. The argument on $U_{\varepsilon}^-$ is completely analogous. As we have pointed out before, this concludes the argument.
	\end{proof}

	For our later applications, we also note the following immediate consequence of \cref{thm:dictBallDisk+} and the commuting diagram in \cref{fig:commuting_diagram}.
	
	\begin{cor}
		Let $1\leq i\leq n-1$, $\bar{f}\in\Dclass^{n-i-1}$, and $f=\bar{f}(\pair{e_n}{{}\cdot{}})$. Then for every subspace $E\in\Gr_{i+1}(\R^n)$ with $e_n\in E$,
		\begin{equation*}
			\varphi_{i,f}(K)
			= \int_{\S^{i}(E)} (\pi_{n-i-1,\BB}\bar{f})(\pair{e_n}{u})\, dS_i^E(K,u),
			\qquad K\in\K(E).
		\end{equation*}
	\end{cor}

	We are now ready to prove \cref{thm:zonalContBallHadwiger}.

	\begin{proof}[Proof of \cref{thm:zonalContBallHadwiger}]
		Let $\varphi \in \Val_i(\R^n)$ be zonal. By \cref{thm:zonalDiskHadwiger}, there exists a zonal function $g=\bar{g}(\pair{e_n}{{}\cdot{}}) \in C(\S^{n-1})$ such that $\varphi=\psi_{i,g}$.
		According to \cref{prop:Tibij}, there exists some $\bar{f} \in \Dclass^{n-i-1}$ such that $\bar{g}=T_{n-i-1}\bar{f}$. Due to \cref{thm:dictBallDisk+}, the function $f=\bar{f}(\pair{e_n}{{}\cdot{}})$ then provides the desired improper integral representation.
		
		For the uniqueness, suppose that $\varphi=0$. Then \cref{thm:zonalDiskHadwiger} implies that $\bar{g}$ is a linear function. According to \cref{prop:Tibij}, the transform $T_{n-i-1}$ is injective, and direct computation shows that it maps linear functions to linear functions. Hence, $\bar{f}$ is a linear function.
	\end{proof}
	
	%%%%%%%%%%%%%%%%%%%%%%%%%%%%%%%%%%%%%%%%%%%%%%%%%%%%%%%%%%%%%%%%%%%%%%%%%%%
	\section{Applications}
	\label{sec:applications}
	%%%%%%%%%%%%%%%%%%%%%%%%%%%%%%%%%%%%%%%%%%%%%%%%%%%%%%%%%%%%%%%%%%%%%%%%%%%
	
	%%%%%%%%%%%%%%%%%%%%%%%%%%%%%%%%%%%%%%%%%%%%%%%%%%%%%%%%%%%%%%%%%%%%%%%%%%%
	\subsection{Integral geometric formulas}
	%%%%%%%%%%%%%%%%%%%%%%%%%%%%%%%%%%%%%%%%%%%%%%%%%%%%%%%%%%%%%%%%%%%%%%%%%%%	
	
	In the following, we apply the Hadwiger type \cref{thm:zonalDiskHadwiger} for zonal valuations, and more specifically, the determination by their values on cones (see \cref{sec:cones}), to prove some integral geometric formulas. First, we establish the additive kinematic formula \cref{eq:zonal_kinematic} in \cref{thm:zonal_kinematic}.
	
	\begin{proof}[Proof of \cref{thm:zonal_kinematic}]
		For convenience, we define functionals $\varphi$ and $\varphi_{i}$ as follows.
		\begin{align*}
			\varphi(K,L)
			&:= \frac{1}{\kappa_{n-1}} \int_{\SO(n-1)} \int_{\S^{n-1}} g(u) \, dS_j(K+\vartheta L,\DD,u)\, d\vartheta,\\
			\varphi_{i}(K,L)
			&:= \frac{1}{\kappa_{n-1}^2} \int_{\S^{n-1}} \int_{\S^{n-1}} \! q(u,v) \, dS_i(K,\DD,u)\, dS_{j-i}(L,\DD,v),
		\end{align*}
		where $q(u,v)=\bar{q}(\pair{e_n}{u},\pair{e_n}{v})$ and $\bar{q}(s,t)=\max\{s,t\}\bar{g}(\min\{s,t\})$.
		Observe that $\varphi$ is a translation-invariant, continuous, and zonal valuation in both of its arguments. The same is true for $\varphi_i$, which is also homogeneous in each argument; that is, $\varphi_i({}\cdot{},L)\in \Val_i(\R^n)$ and $\varphi_i(K,{}\cdot{})\in\Val_{j-i}(\R^n)$. Thus, by \cref{cor:zonalValDetByCones} and the homogeneous decomposition theorem by McMullen~\cite{McMullen1980}, it suffices to show that for all $s,t\in [-1,1]\setminus\{0\}$ and $\lambda,\mu\geq 0$,
		\begin{equation}\label{eq:zonal_kinematic:proof}
			\varphi(\lambda C_s, \mu C_t)
			= \sum_{i=0}^{j} \binom{j}{i}\lambda^{i}\mu^{j-i}\varphi_i(C_s,C_t).
		\end{equation}

		\begin{figure}
			\captionsetup[subfigure]{labelformat=empty}
    		\centering
    		\begin{subfigure}[b]{0.30\textwidth}
        		\includegraphics[width=\textwidth]{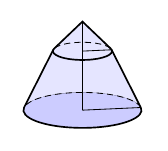}
        		\vspace{-1cm}
       			\caption{$0 < s \leq t$}
        		\label{fig:same}
    	\end{subfigure}
    	\begin{subfigure}[b]{0.30\textwidth}
        	\includegraphics[width=\textwidth]{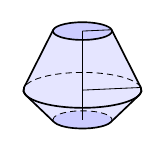}
        	\vspace{-1cm}
        	\caption{$t<0<s$}
       	 	\label{fig:opposite}
    	\end{subfigure}
    	\vspace{-3mm}
   		\caption{The Minkowski sum $\lambda C_s + \mu C_t$.}\label{fig:cones}
   		\vspace{-5mm}
		\end{figure}

		First, we consider the case where both cones are pointing in the same direction. To this end, let $0<s\leq t$. Then the Minkowski sum $\lambda C_s + \mu C_t$ is the cone $(\lambda+\mu)C_s$, truncated and glued together with a translate of $\mu C_t$ (see \cref{fig:cones}). More precisely, 
		\begin{equation*}
			\lambda C_s+ \mu C_t
			= \big((\lambda+\mu)C_s \setminus  (\mu C_s + \lambda\tfrac{\sqrt{1-s^2}}{s} e_n)\big) \cup (\mu C_t + \lambda \tfrac{\sqrt{1-s^2}}{s} e_n).
		\end{equation*}
		From \cref{eq:area_meas_cone_disk}, the valuation property, homogeneity, and translation invariance of $\psi_{j,g}$ (see the definition at the beginning of \cref{sec:moving_between_integral_reps}), the left hand side of \cref{eq:zonal_kinematic:proof} becomes
		\begin{align*}
			\varphi(\lambda C_s,\mu C_t)
			&= \frac{1}{\kappa_{n-1}}\big(\psi_{j,g}((\lambda+\mu)C_s) - \psi_{j,g}(\mu C_s) + \psi_{j,g}(\mu C_t)\big) \\
			&= ((\lambda+\mu)^j - \mu^j) \bigg(\bar{g}(-1) + \frac{\bar{g}(s)}{s}\bigg) + \mu^j \bigg(\bar{g}(-1)+ \frac{\bar{g}(t)}{t}\bigg).
		\end{align*}
		For the right hand side, applying again \cref{eq:area_meas_cone_disk}, 
		\begin{equation*}
			\varphi_i(C_s,C_t)
			= \left\{\begin{array}{ll}
				\displaystyle \bar{q}(-1,-1)+ \bar{q}(1,-1) + \frac{\bar{q}(-1,t)}{t}  + \frac{\bar{q}(1,t)}{t},
				&i=0,\vspace{2mm}\\
				\displaystyle \bar{q}(-1,-1) + \frac{\bar{q}(s,-1)}{s} + \frac{\bar{q}(-1,t)}{t} + \frac{\bar{q}(s,t)}{st},
				&0<i< j, \vspace{2mm}\\
				\displaystyle \bar{q}(-1,-1) + \frac{\bar{q}(s,-1)}{s} +\bar{q}(-1,1)+ \frac{\bar{q}(s,1)}{s} ,
				&i=j.
			\end{array}\right.
		\end{equation*}
		Plugging in the definition of $\bar{q}$ in terms of $\bar{g}$, one readily verifies \cref{eq:zonal_kinematic:proof}. If $0<t\leq s$, the argument is analogous.
		
		Next, we consider the case where the two cones have opposite orientations. To this end, let $t<0<s$. Then the Minkowski sum $\lambda C_s + \mu C_t$ consists of two truncated cones glued together (see \cref{fig:cones}). More precisely,		
		\begin{align*}
			\lambda C_s + \mu C_t
			&= \big((\lambda+\mu)C_s \setminus (\mu C_s + \lambda\tfrac{\sqrt{1-s^2}}{s}e_n)\big) \cup (\mu\DD + \lambda\tfrac{\sqrt{1-s^2}}{s}e_n) \\
			&\qquad \cup \big((\lambda+\mu)C_t \setminus (\lambda C_t + \mu \tfrac{\sqrt{1-t^2}}{t}e_n)\big) \cup (\lambda \DD + \mu \tfrac{\sqrt{1-t^2}}{t}e_n).
		\end{align*}
		Hence, similarly as before, the left hand side of \cref{eq:zonal_kinematic:proof} amounts to
		\begin{align*}
			\varphi(\lambda C_s, \mu C_t)
			= ((\lambda+\mu)^j - \mu^j) \frac{\bar{g}(s)}{s} + \mu^j \bar{g}(1)  + ((\lambda+\mu)^j - \lambda^j) \frac{\bar{g}(t)}{|t|} + \lambda^j\bar{g}(-1).
		\end{align*}
		For the right hand side, applying \cref{eq:area_meas_cone_disk},
		\begin{equation*}
			\varphi_i(C_s,C_t)
			= \left\{\begin{array}{ll}
				\displaystyle \bar{q}(-1,1)+ \bar{q}(1,1) + \frac{\bar{q}(-1,t)}{|t|}  + \frac{\bar{q}(1,t)}{|t|},
				&i=0,\vspace{2mm}\\
				\displaystyle \bar{q}(-1,1) + \frac{\bar{q}(s,1)}{s} + \frac{\bar{q}(-1,t)}{|t|} + \frac{\bar{q}(s,t)}{s|t|},
				&0<i< j, \vspace{2mm}\\
				\displaystyle \bar{q}(-1,1) + \frac{\bar{q}(s,1)}{s} +  \bar{q}(-1,-1) +\frac{\bar{q}(s,-1)}{s},
				&i=j.
			\end{array}\right.
		\end{equation*}
		Plugging in the definition of $\bar{q}$ in terms of $\bar{g}$, one readily verifies \cref{eq:zonal_kinematic:proof}. If $s<0<t$, the argument is analogous. Thus, we have shown \cref{eq:zonal_kinematic:proof} for all $s,t\in [-1,1]\setminus\{0\}$ and $\lambda,\mu\geq 0$, which concludes the proof.		 
	\end{proof}
	
	Note that the choice of the function $\bar{q}$ in the proof of \cref{thm:zonal_kinematic} is far from unique. In fact, one can add any function of the form $\bar{q}_1(s)t + s \bar{q}_2(t) + c\cdot st$, where $\bar{q}_1,\bar{q}_2\in C[-1,1]$ and $c\in\R$.

	We now turn to the Kubota-type formula \cref{eq:zonal_Kubota_polarized}. First, we prove a version of this formula involving intrinsic volumes of projections.
	
	\begin{thm}\label{thm:zonal_Kubota}
		Let $1\leq i\leq n-1$. Then for all $K\in\K(\R^n)$,
		\begin{equation}\label{eq:zonal_Kubota}
			\int_{\Gr_{i}(\R^n,e_n)} V_i(K|E)\, dE
			= \frac{n\kappa_{i-1}}{i\kappa_{n-1}} V(K^{[i]},\DD^{[n-i]}).
		\end{equation}
	\end{thm}
	\begin{proof}
		Observe that both sides define zonal valuations in $\Val_i(\R^n)$. Therefore, according to \cref{cor:zonalValDetByCones}, it suffices to show the identity on the family $C_s$ of cones for $s\in[-1,1]\setminus\{0\}$. According to \cref{eq:area_meas_cone_disk}, for $1\leq i\leq n$,
		\begin{equation*}
			V(C_s^{[i]},\DD^{[n-i]})
			= \frac{1}{n}\int_{\S^{n-1}} h_{C_s}(u)\, dS(C_s^{[i-1]},\DD^{[n-i]},u)
			= \frac{\kappa_{n-1}}{n} \frac{\sqrt{1-s^2}}{\abs{s}}.
		\end{equation*}
		Clearly, the orthogonal projection of $C_s$ onto $E\in\Gr_i(\R^n,e_n)$ is precisely the cone with base $\DD\cap E$ and apex $\tfrac{\sqrt{1-s^2}}{s}e_n$, so
		\begin{equation*}
			\int_{\Gr_{i}(\R^n,e_n)} V_i(C_s|E)\, dE
			%= \int_{\Gr_{i}(\R^n,e_n)} \frac{\kappa_{i-1}}{i} \frac{\sqrt{1-s^2}}{\abs{s}} \, dE
			= \frac{\kappa_{i-1}}{i} \frac{\sqrt{1-s^2}}{\abs{s}}.
		\end{equation*}
		Hence, \cref{eq:zonal_Kubota} holds for all $C_s$, which, by \cref{cor:zonalValDetByCones}, concludes the proof.
	\end{proof}
	
	By a classical polarization argument, we obtain from \cref{eq:zonal_Kubota} that for all functions $f\in C(\S^{n-1})$ and convex bodies $K_1,\ldots,K_{i-1} \in\K(\R^n)$,
	\begin{multline*}
		\int_{\Gr_{i}(\R^n,e_n)} \int_{\S^{i-1}(E)}f(u)\, dS^E(K_1|E,\dots,K_{i-1}|E,u) \, dE \\
		= \frac{\kappa_{i-1}}{\kappa_{n-1}} \int_{\S^{n-1}} f(u)\, dS(K_1,\ldots,K_{i-1},\DD^{[n-i]},u).
	\end{multline*}
	This is the formulation of the Kubota-type formula in \cite{Hug2024}*{Theorem~3.2}. In particular, by setting $K_1=\cdots=K_{i-1}=K$, we obtain \cref{eq:zonal_Kubota_polarized}.
%	
%	By polarization we obtain the following
%	\begin{equation*}
%		\int_{\Gr_{i}(\R^n,e_n)} V_i(K_1|E,\ldots,K_i|E)\, dE
%		= \frac{n\kappa_{i-1}}{i\kappa_{n-1}} V(K_1,\ldots,K_i,\DD^{[n-i]}).
%	\end{equation*}
%	In particular, we recover .
%	
%	Maybe a nicer formulation is the following:
%	\begin{align*}
%		\int_{\Gr_{i}(\R^n,e_n)} \int_{\S^{i-1}(E)} \! f(u)\, dS^E_{i-1}(K|E,u) \, dE
%		= \frac{\kappa_{i-1}}{\kappa_{n-1}} \int_{\S^{n-1}} \! f(u)\, dS_i(K^{[i-1]},\DD^{[n-i]},u)
%	\end{align*}

	Next, we want to deduce the Crofton-type formula \cref{eq:mean_sec_disk_zonal}. To this end, we 
	recall the following formula about integration over affine Grassmannians.
	If $1\leq j\leq n$ and $F\in\AGr_k(\R^n)$ is some fixed affine subspace, where $n-j\leq k\leq n$, then for every measurable function ${\xi}:\AGr_{j+k-n}(F)\to [0,\infty)$,
	\begin{equation}\label{eq:AGr_intersect}
		\int_{\AGr_{j}(\R^n)} {\xi}(E\cap F) \, dE
		= \frac{\omega_{j+k-n+1}\omega_{n+1}}{\omega_{j+1}\omega_{k+1}} \int_{\AGr_{j+k-n}(F)} {\xi}(E)\, dE.
	\end{equation}
	This follows from the uniqueness of the invariant measure on $\AGr_{j+k-n}(F)$, where the multiplicative constant can be computed from the classical Crofton formula (see, e.g., \cite{Goodey2014}*{p.~481}).  
	We require the following integral identity.
	
	\begin{lem}\label{lem:AGr_intersect_particular}
		Let $1\leq j\leq n$ and $\xi: \AGr_{j-1}(\R^n)\to [0,\infty)$ be measurable. Then
		\begin{align*}
			\int_{\AGr_{j}(\R^n)} \int_{\R} \xi(E\cap (e_n^\perp + te_n)) \, dt \, dE
			= \frac{\omega_j \omega_{n+1} }{\omega_{j+1} \omega_n} \int_{ \Gr_{j-1}(e_n^\perp)} \int_{ E^{\perp}}  \xi(E+x) \, dx \, dE.
		\end{align*}
	\end{lem}
	\begin{proof}
		As an instance of \cref{eq:AGr_intersect}, with $k=n-1$ and $F=e_n^\perp + te_n$,
		\begin{align*}
			& \frac{\omega_{j+1} \omega_n}{\omega_j \omega_{n+1} } \int_{\AGr_{j}(\R^n)} \xi(E\cap (e_n^\perp + te_n)) \, dE
			= \int_{\AGr_{j-1}(e_n^\perp + te_n)} \xi(E) \, dE \\
			&\qquad = \int_{\AGr_{j-1}(e_n^\perp)} \xi(E + te_n) \, dE
			= \int_{\Gr_{j-1}(e_n^\perp)} \int_{ E^{ \perp {(e_n^\perp)} } } \xi(E + y + te_n) \, dy \, dE,
		\end{align*}
		where the final equality is by the uniqueness of the invariant measure on $\AGr_{j-1}(e_n^\perp)$ and $E^{ \perp {(e_n^\perp)} }$ denotes the orthogonal complement of $E$ relative to $e_n^\perp$. Hence, 
		\begin{align*}
			&\int_{\AGr_{j}(\R^n)} \int_{\R} \xi(E\cap (e_n^\perp + te_n)) \, dt \, dE \\
			&\qquad = \frac{\omega_j \omega_{n+1} }{\omega_{j+1} \omega_n} \int_{\Gr_{j-1}(e_n^\perp)} \int_{ E^{ \perp {(e_n^\perp)} } } \int_{\R} \xi(E + y + te_n) \, dt \, dy \, dE \\
			&\qquad = \frac{\omega_j \omega_{n+1} }{\omega_{j+1} \omega_n} \int_{\Gr_{j-1}(e_n^\perp)} \int_{ E^{ \perp } } \xi(E + x) \, dx \, dE,
		\end{align*}
		where we applied Fubini's theorem.
	\end{proof}
	
	\begin{proof}[Proof of \cref{cor:mean_sec_disk_zonal}]
		First, note that for every convex body $C\in\K(\R^n)$,
		\begin{align*}
			h_{C}(e_n) + h_C(-e_n)
			= V_1(C|\mathrm{span}\, e_n)
			= \int_{\R} V_0(C\cap (e_n^\perp + t e_n))\, dt.
		\end{align*}
		Consequently, since $K$ is origin symmetric, we have that
		\begin{align*}
			&\int_{\AGr_{j}(\R^n)} h_{K\cap E}(e_n) \, dE
			= \frac{1}{2} \int_{\AGr_{j}(\R^n)} \big(h_{K\cap E}(e_n) + h_{K\cap E}(-e_n)\big) \, dE \\
			&\qquad = \frac{1}{2} \int_{\AGr_{j}(\R^n)} \int_{\R} V_0(K\cap E\cap (e_n^\perp + te_n)) \, dt \, dE.
		\end{align*}
		Applying \cref{lem:AGr_intersect_particular} to $\xi(E):=V_0(K\cap E)$ yields		
		\begin{align*}
			& \frac{\omega_{j+1} \omega_n}{\omega_j \omega_{n+1} } \int_{\AGr_{j}(\R^n)} h_{K\cap E}(e_n) \, dE
			= \frac{1 }{2 } \int_{\Gr_{j-1}(e_n^\perp)} \int_{E^\perp} V_0(K\cap (E+x)) \, dx \, dE \\
			&\qquad = \frac{1 }{2 } \int_{\Gr_{j-1}(e_n^\perp)} V_{n-j+1}(K|E^\perp) \, dE
			= \frac{1 }{2 } \int_{\Gr_{n-j+1}(\R^n,e_n)} V_{n-j+1}(K|E)\, dE.
			%= V(K^{[i]},\DD^{[n-i]})
		\end{align*}
		By combining this with the Kubota-type formula \cref{eq:zonal_Kubota}, we obtain \cref{eq:mean_sec_disk_zonal}.
	\end{proof}
	
	%%%%%%%%%%%%%%%%%%%%%%%%%%%%%%%%%%%%%%%%%%%%%%%%%%%%%%%%%%%%%%%%%%%%%%%%%%%
	\subsection{Mean section operators}
	%%%%%%%%%%%%%%%%%%%%%%%%%%%%%%%%%%%%%%%%%%%%%%%%%%%%%%%%%%%%%%%%%%%%%%%%%%%
	
	We now turn to the application of \cref{cor:zonalValDetByRestr} to the mean section operators. Our aim is to deduce \cref{MSO_Berg_fct} for $j>2$ from the instance where $j=2$.
	In our argument, we use the fact that for a convex body of some linear subspace $E$, its Steiner point relative to $E$ agrees with its Steiner point relative to the ambient space (cf.~\cite{Schneider2014}*{p.~315}). 
	
	We require the following relation between Berg's functions.
	
	\begin{lem}\label{Berg_fcts_relation}
		For every $j>2$, there exists $a_j \in \R$ such that
		\begin{equation}\label{eq:Berg_fcts_relation}
			\pi_{j-2,\BB}g_j(s)
			%= \omega_{j} \int_0^1 g_j(st)(1-t^2)^{\frac{j-4}{2}}dt
			= (j-1) g_2(s) + a_j s,
			\qquad s\in (-1,1).
		\end{equation}
	\end{lem}
	\begin{proof}
		Take $E\in\Gr_2(\R^n)$ with $e_n\in E$ and $K\in\K(E)$. Then, by combining Berg's identity \cref{eq:Bergs_result} with \cref{sph_proj_zonal} and \cref{eq:mixed_area_meas_lift:intro}, we have that
		\begin{align*}
			& h_{K-s(K)}(e_n)
			= \int_{\S^{n-1}} g_n(\pair{e_n}{u})\, dS_1(K,u) \\
			&\qquad = \frac{1}{n-1} \int_{\S^1(E)} [\pi_{n-2,\BB}g_n](\pair{e_n}{u}) \, dS_1^E(K,u).
		\end{align*}
		On the other hand, applying \cref{eq:Bergs_result} in the subspace $E$ yields
		\begin{equation*}
			h_{K-s(K)}(e_n)
			= \int_{\S^{1}(E)} g_2(\pair{e_n}{u})\, dS_1^E(K,u).
		\end{equation*}
		Since $K\in\K(E)$ was arbitrary, $\pi_{n-2,\BB}g_n - (n-1)g_2$ is a linear function.
	\end{proof}
	
	Note that the lemma yields the following integral identity, which might be of independent interest. For every $j>2$, there exists $a_j \in \R$ such that 
	\begin{equation*}
		\omega_j \int_0^1 g_j(st)(1-t^2)^{\frac{j-4}{2}}dt
		=  (j-1) g_2(s) + a_j s,
		\qquad s\in (-1,1).
	\end{equation*}

	We are now ready to recover \cite{Goodey2014}*{Theorem~4.4}.
	Let us remark that the existence of the integrals on the right hand side of \cref{eq:Bergs_result} and \cref{eq:MSO_Berg_fct} was shown in \cite{Berg1969} and \cite{Goodey2014} using a certain regularization procedure. More recently, Knoerr~\cite{Knoerr2024} gave a simpler argument that these integrals exist and also define continuous valuations. 
	
	\begin{proof}[Proof of \cref{cor:MSOcase2implJ}]%[Proof of \cref{MSO_Berg_fct} for $j>2$]
		Let $2<j<n$ and write $i = n-j+1$. Due to the rotational equivariance, it suffices to show the claim only for $u=e_n$. Let $E\in\Gr_{i+1}(E)$ with $e_n\in E$ and take $K\in\K(E)$. As a consequence of \cref{eq:AGr_intersect}, $\MSO_j K= \frac{\omega_3\omega_{n+1}}{\omega_{j+1}\omega_{n-j+3}} \MSO_2^E K$, where $\MSO_2^E$ denotes the mean section operator relative to $E$ (see \cite{Goodey2014}*{Lemma~3.3}). Consequently, by an application of \cref{MSO_Berg_fct} for $j=2$,
		\begin{align*}
			&h_{\MSO_j(K-s(K))}(e_n)
			= \frac{\omega_3\omega_{n+1}}{\omega_{j+1}\omega_{n-j+3}} h_{\MSO_2^E(K-s(K))}(e_n) \\
			&\qquad = \frac{\omega_3\omega_{n+1}}{\omega_{j+1}\omega_{n-j+3}} c_{n-j+2,2} \int_{\S^{i}(E)} g_2(\pair{e_n}{u})\, dS_{i}^E(K,u).
		\end{align*}
		Moreover, by relation \cref{eq:Berg_fcts_relation}, we have that
		\begin{align*}
			\int_{\S^{i}(E)} g_2(\pair{e_n}{u})\, dS_{i}^E(K,u)
			&= \frac{1}{j-1} \int_{\S^{i}(E)} [\pi_{j-2,\BB}g_j](\pair{e_n}{u})\, dS_{i}^E(K,u) \\
			&= \frac{\binom{n-1}{n-j+1}}{j-1} \int_{\S^{n-1}} g_j(\pair{e_n}{u})\, dS_{i}(K,u)
		\end{align*}
		where the final equality is due to \cref{sph_proj_zonal} and \cref{eq:mixed_area_meas_lift:intro}. Finally, note that $h_{\MSO_j(K-s(K))}(e_n)$ defines a zonal valuation in $\Val_i(\R^n)$ and the final integral expression does too, as was shown in \cite{Knoerr2024}*{Section~3}. Our argument shows that they coincide on $E$, so by \cref{cor:zonalValDetByRestr}, they coincide on all convex bodies in $\R^n$.
	\end{proof}

	%%%%%%%%%%%%%%%%%%%%%%%%%%%%%%%%%%%%%%%%%%%%%%%%%%%%%%%%%%%%%%%%%%%%%%%%%%%
	\appendix
	%%%%%%%%%%%%%%%%%%%%%%%%%%%%%%%%%%%%%%%%%%%%%%%%%%%%%%%%%%%%%%%%%%%%%%%%%%%
	
	%%%%%%%%%%%%%%%%%%%%%%%%%%%%%%%%%%%%%%%%%%%%%%%%%%%%%%%%%%%%%%%%%%%%%%%%%%%
	\section{The transform \texorpdfstring{$R_{a,b}$}{R a,b}}
	\label{app:Rabtrans}
	%%%%%%%%%%%%%%%%%%%%%%%%%%%%%%%%%%%%%%%%%%%%%%%%%%%%%%%%%%%%%%%%%%%%%%%%%%%
	
	In this section, we consider the following integral transform that comes up naturally when dealing with restrictions of zonal valuations to proper subspaces.
	
	\begin{defi}
		For $a,b>0$, we define $R_{a,b}: C(-1,1) \to C(-1,1)$ by
		\begin{equation*}
			(R_{a,b}\bar{f})(t)
			:= \int_0^1 \bar{f}(st) s^{a-1} (1-s^2)^{b-1} ds,
			\qquad t\in(-1,1).
		\end{equation*}
	\end{defi}
	
	In order to see that the transform $R_{a,b}$ is well-defined, note that for every compact set $I\subseteq (-1,1)$, there exists $C>0$ such that $\max_{s\in [0,1]} |f(st)| \leq C$ for all $t\in I$. Thus, the integral exists and by dominated convergence, $R_{a,b}f\in C(-1,1)$.
	We want to show that the map is injective, for which we require two lemmas.
	
%	Moreover, by a change of variables,
%	\begin{equation*}
%		(R_{a,b}f)(t)
%		= \frac{1}{t^{a+2b-2}}\int_0^t f(s)s^{a-1}(t^2-s^2)^{b-1}ds,
%		\qquad t\in (-1,1).
%	\end{equation*}
	
	\begin{lem}\label{lem:derRab}
		Let $a>0$, $b\geq 1$, and $f\in C(-1,1)$. Then for all $t\in (-1,1)\setminus\{0\}$, the function $|t|^aR_{a,b}f(t)$ is differentiable at $t$ and
		\begin{equation}\label{eq:derRab}
			\frac{d}{dt}\left[ |t|^a (R_{a,b}f)(t)\right]
			= \begin{cases}
				|t|^{a-1}f(t),					&b=1, \\
				2(b-1)t^{-3} (R_{a+2,b-1}f)(t),	&b>1.
			\end{cases}
		\end{equation}
	\end{lem}
	\begin{proof}
		By a change of variables,
		\begin{equation*}
			|t|^a(R_{a,b}\bar{f})(t)
			= \int_0^t \bar{f}(s)|s|^{a-1}\bigg(1-\frac{s^2}{t^2}\bigg)^{\! b-1}ds.
		\end{equation*}
		For $b=1$, the claim follows directly from the fundamental theorem of calculus. For $b>1$, note that at every $t\in(-1,1)\setminus\{0\}$ and $s\in [0,t]$ (or $[t,0]$, respectively), the partial derivative of $(1-s^2/t^2)^{b-1}$ with respect to $t$ exists. Hence, the Leibniz integral rule implies that $R_{a,b}\bar{f}$ is differentiable at $t$ and
		\begin{align*}
			&\frac{d}{dt} \left[ |t|^a (R_{a,b}\bar{f})(t)\right]
			= \bar{f}(t)|t|^{a-1}\bigg(1-\frac{t^2}{t^2}\bigg)^{\! b-1} + \int_0^{t} \bar{f}(s)|s|^{a-1}\frac{\partial}{\partial t}\bigg(1-\frac{s^2}{t^2}\bigg)^{\! b-1} ds \\
			&\qquad = 2(b-1)t^{-3}\int_0^t \bar{f}(s)|s|^{a+1}\bigg(1-\frac{s^2}{t^2}\bigg)^{\! b-2} ds
			= 2(b-1)t^{-3}(R_{a+2,b-1}\bar{f})(t),
		\end{align*}
		which yields the claim for $b>1$.
	\end{proof}

	\begin{lem}\label{lem:compRabs}
		Let $a_1,b_1,a_2,b_2>0$ and suppose that $a_1=a_2+2b_2$. Then
		\begin{equation*}
			R_{a_1,b_1}R_{a_2,b_2}
			= \frac{1}{2} \mathrm{B}(b_1,b_2) R_{a_2,b_1+b_2}.
		\end{equation*}
	\end{lem}
	\begin{proof}
		For every $\bar{f}\in C(-1,1)$ and $t\in (-1,1)$, we have that
		\begin{equation*}
			(R_{a_1,b_1}R_{a_2,b_2}\bar{f})(t)
			= \int_0^1 \int_0^1 \bar{f}(x st) x^{a_2-1} (1-x^2)^{b_2-1} dx~ s^{a_1-1}(1-s^2)^{b_1-1} ds.
		\end{equation*}
		By applying the change of variables $r=sx$ to the inner integral, we obtain
		\begin{align*}
			&(R_{a_1,b_1}R_{a_2,b_2}\bar{f})(t)
			= \int_0^1 \int_0^s \bar{f}(rt) \frac{r^{a_2-1}}{s^{a_2-1}} \bigg(1-\frac{r^2}{s^2}\bigg)^{b_2-1} dr~ s^{a_1-1}(1-s^2)^{b_1-1} \frac{1}{s} ds \\
			&\qquad = \int_0^1  \bar{f}(rt) r^{a_2-1} \int_r^1 s^{a_1-a_2-1} \bigg(1-\frac{r^2}{s^2}\bigg)^{b_2-1}(1-s^2)^{b_1-1} ds ~dr,
		\end{align*}
		where the second equality is due to a change of the order of integration. It remains to compute the inner integral. To that end, we rearrange the integrand using the identity $a_1=a_2+2b_2$ and then apply the change of variables $y=(1-s^2)/(1-r^2)$, which yields
		\begin{align*}
			&\int_r^1 s^{a_1-a_2-1} \bigg(1-\frac{r^2}{s^2}\bigg)^{b_2-1}(1-s^2)^{b_1-1} ds
			=  \int_r^1 (s^2-r^2)^{b_2-1} (1-s^2)^{b_1-1}sds \\
			&\qquad = \frac{1}{2}(1-r^2)^{b_1+b_2-1} \int_0^1 y^{b_1-1}(1-y)^{b_2-1} dy
			= \frac{1}{2} \mathrm{B}(b_1,b_2)  (1-r^2)^{b_1+b_2-1},
		\end{align*}
		where $\mathrm{B}({}\cdot{},{}\cdot{})$ denotes the classical beta function. Plugging this into the expression above yields the desired result.
	\end{proof}
	
	\begin{prop}\label{lem:Rabinj}
		For each $a,b>0$, the map $R_{a,b}$ is injective.
	\end{prop}
	\begin{proof}
		First note, that by \cref{lem:derRab}, $R_{a,1}$ is injective for every $a>0$. Next, if $0<b<1$, \cref{lem:compRabs} applied with $(a_2,b_2) = (a,b)$ and $(a_1,b_1) = (a+2b,1-b)$ implies that
		\begin{align*}
			R_{a+2b,1-b}R_{a,b}
			= \frac{1}{2} \mathrm{B}(1-b,b) R_{a,1}.
		\end{align*}
		Consequently, using that the beta function takes always positive values on real arguments, we deduce that $R_{a,b}$ is injective for every $a>0$ and $0<b<1$, since $R_{a,1}$ is. \cref{lem:derRab} shows that $R_{a,b}$ is injective whenever $R_{a+2,b-1}$ is. Finally, we complete the proof by induction.
	\end{proof}
	
	Next, we want to show that the transform $R_{a,b}$ maps the space $C^\infty[-1,1]$ into itself bijectively. 	
	
	\begin{lem}\label{lem:Rab_smooth2smooth}
		For $a,b>0$, the map $R_{a,b}$ maps the space $C^\infty[-1,1]$ into itself.
	\end{lem}
	\begin{proof}
		Take $\bar{f}\in C^\infty[-1,1]$ and note that for $k\geq 0$ and $t\in(-1,1)$,
		\begin{equation*}
			\frac{d^k}{dt^k} R_{a,b}\bar{f}(t)
			= \int_0^1 \bar{f}^{(k)}(st)s^{a+k-1}(1-s^2)^{b-1}ds.
		\end{equation*}
		This follows by induction on $k\geq 0$ by interchanging integral and derivative, using the fact that $\bar{f}$ is infinitely differentiable and all of its derivatives are bounded. Thus, $R_{a,b}\bar{f}$ is a $C^\infty(-1,1)$ function. By interchanging the above integral with the limit $\lim_{t\to \pm 1}$, we see that all higher order derivatives of $R_{a,b}\bar{f}$ extend continuously to $[-1,1]$, that is, $R_{a,b}\bar{f}\in C^\infty[-1,1]$.
	\end{proof}
	
	We want to show that the transform $R_{a,b}$ also maps the space $C^\infty[-1,1]$ onto itself. We define the right candidate for the inverse operator recursively. 
	
	\begin{defi}
		For $a,b>0$, we define $Q_{a,b}: C^\infty[-1,1]\to C^\infty[-1,1]$ recursively as follows. Let $\bar{g}\in C^\infty[-1,1]$.
		\begin{enumerate}[label=(\roman*)]
			\item If $b=1$, then $Q_{a,1}\bar{g}(s) := a\bar{g}(s)+s\bar{g}'(s)$, $s\in [-1,1]$.
			\item If $0<b<1$, then $\displaystyle Q_{a,b}\bar{g} := \frac{2}{B(b,1-b)}Q_{a,1}R_{a+2b,1-b}\bar{g}$.
			\item If $b>1$, then $\displaystyle Q_{a,b}\bar{g} := \frac{1}{2(b-1)}Q_{a,b-1}Q_{a+2b-2,1}\bar{g}$.
		\end{enumerate}
	\end{defi}
	
	From \cref{lem:Rab_smooth2smooth} and an inductive argument, it is immediate that $Q_{a,b}$ is well-defined as a transform mapping the space $C^\infty[-1,1]$ into itself.
	
	\begin{prop}\label{lem:Rabsurj}
		For each $a,b>0$, the map $R_{a,b}:C^\infty[-1,1]\to C^\infty[-1,1]$ is a bijection and $R_{a,b}^{-1}=Q_{a,b}$.
	\end{prop}
	\begin{proof}
		We want to show inductively that $R_{a,b}Q_{a,b}\bar{g}=\bar{g}$ for all $a,b>0$ and $\bar{g}\in C^\infty[-1,1]$.
		If $b=1$, then
		\begin{align*}
			&[R_{a,1}Q_{a,1}\bar{g}](t)
			= \int_0^1 [Q_{a,1}\bar{g}](st)s^{a-1}ds
			= \int_0^1 \left( a\bar{g}(st) + st\bar{g}'(st) \right) s^{a-1} ds \\
			&\qquad = \int_0^1 \frac{d}{ds}\left[\bar{g}(st)s^a\right]ds
			= \bar{g}(t)
		\end{align*}
		for all $t\in[-1,1]$, and thus, $R_{a,1}Q_{a,1}\bar{g}=\bar{g}$. For $0<b<1$, by definition
		\begin{equation*}
			Q_{a,b}\bar{g}
			= \frac{2}{B(b,1-b)}Q_{a,1}R_{a+2b,1-b}\bar{g}.
		\end{equation*}
		Applying the operator $R_{a+2b,1-b}R_{a,b}$ on both sides yields
		\begin{align*}
			&R_{a+2b,1-b}R_{a,b}Q_{a,b}\bar{g}
			= \frac{2}{B(b,1-b)}R_{a+2b,1-b}R_{a,b}Q_{a,1}R_{a+2b,1-b}\bar{g} \\
			&\qquad = R_{a,1}Q_{a,1}R_{a+2b,1-b}\bar{g}
			= R_{a+2b,1-b}\bar{g},
		\end{align*}
		where the second equality is an instance of \cref{lem:compRabs} and the final equality is due to the previous step. According to \cref{lem:Rabinj}, the map $R_{a+2b,1-b}$ is injective, which implies that $R_{a,b}Q_{a,b}\bar{g}=\bar{g}$.
		
		For the induction step, let $b>1$ and note that $R_{a,b}=2(b-1)R_{a+2b-2,1}R_{a,b-1}$ due to \cref{lem:compRabs}. Therefore, by the recursive definition of $Q_{a,b}$,
		\begin{equation*}
			R_{a,b}Q_{a,b}\bar{g}
			= R_{a+2b-2,1}R_{a,b-1}Q_{a,b-1}Q_{a+2b-2,1}\bar{g}
			= R_{a+2b-2,1}Q_{a+2b-2,1}\bar{g}
			= \bar{g},
		\end{equation*}
		where the second equality is due to the induction hypothesis, and the final equality is due to the case where $b=1$. In conclusion, $R_{a,b}Q_{a,b}\bar{g}=\bar{g}$ for all $a,b>0$.
		
		Finally, for all $a,b>0$ and $\bar{f}\in C^\infty[-1,1]$,
		\begin{equation*}
			R_{a,b}(Q_{a,b}R_{a,b}\bar{f})
			= R_{a,b}Q_{a,b}(R_{a,b}\bar{f})
			= R_{a,b}\bar{f},
		\end{equation*}
		and since $R_{a,b}$ is injective by \cref{lem:Rabinj}, we also have that $Q_{a,b}R_{a,b}\bar{f}=\bar{f}$.
	\end{proof}
	
	%%%%%%%%%%%%%%%%%%%%%%%%%%%%%%%%%%%%%%%%%%%%%%%%%%%%%%%%%%%%%%%%%%%%%%%%%%%
	\section{The transform \texorpdfstring{$T_{\alpha}$}{T alpha}}
	%%%%%%%%%%%%%%%%%%%%%%%%%%%%%%%%%%%%%%%%%%%%%%%%%%%%%%%%%%%%%%%%%%%%%%%%%%%
	
	In this section, we investigate the transform $T_{\alpha}$ from \cref{defi:T_alpha}.
	
	\begin{lem}\label{lem:T_convergence}
		Let $\alpha\geq 0$, $\bar{f}\in\Dclass^\alpha$, and let $\bar{\eta}_{\varepsilon} \in C[-1,1]$, $\varepsilon>0$, be a family of bump functions such that
		\begin{equation*}
			\bar{\eta}_{\varepsilon}(s) = 1 \text{ for } \abs{s}\leq 1-\varepsilon,
			\qquad \bar{\eta}_{\varepsilon}(s) = 0 \text{ for } \abs{s}\geq 1-\tfrac{\varepsilon}{2},
			\qquad \text{and} \qquad 0\leq \bar{\eta}_{\varepsilon} \leq 1.
		\end{equation*}
		Then $T_\alpha(\bar{\eta}_{\varepsilon}\bar{f}) \to T_{\alpha}\bar{f}$ uniformly on $[-1,1]$ as $\varepsilon\to 0^+$.
	\end{lem}
	\begin{proof}
		We define $\bar{\zeta}_{\varepsilon}:=1-\bar{\eta}_{\varepsilon}$ and will show that $T_\alpha(\bar{\zeta}_{\varepsilon}\bar{f}) \to 0$ uniformly on $[-1,1]$.
		For all $s\in [-1,1]$,
		\begin{equation*}
			T_\alpha(\bar{\zeta}_{\varepsilon}\bar{f})(s)
			= (1-s^2)^{\frac{\alpha}{2}}\bar{\zeta}_{\varepsilon}(s)\bar{f}(s) + \alpha s\int_{0}^s \bar{\zeta}_{\varepsilon}(t)\bar{f}(t) (1-t^2)^{\frac{\alpha-2}{2}}dt.
		\end{equation*}
		For the first term, observe that for all $s\geq 0$,
		\begin{equation*}
			\left| (1-s^2)^{\frac{\alpha}{2}}\bar{\zeta}_{\varepsilon}(s)\bar{f}(s) \right|
			\leq \sup_{x\geq 1-\varepsilon} \abs{\bar{f}(x)}(1-x^2)^{\frac{\alpha}{2}}.
			%xrightarrow[\varepsilon\to0^+]{} 0
		\end{equation*}
		The right hand side is independent of $s$, and since $\lim_{s\to1}(1-s^2)^{\frac{\alpha}{2}}\bar{f}(s)=0$, it tends to zero as $\varepsilon\to 0^+$.
		We now turn to the integral expression. For $|s| \leq 1-\varepsilon$, it vanishes. For $1-\varepsilon\le s \leq 1-\frac{\varepsilon}{2}$, by the mean value theorem, there exists $s_0\in (1-\varepsilon,1-\frac{\varepsilon}{2})$ such that
		\begin{align*}
			\alpha s\int_{1-\varepsilon}^{s} \bar{\zeta}_{\varepsilon}(t)\bar{f}(t) (1-t^2)^{\frac{\alpha-2}{2}}dt
			= \bar{f}(s_0) \alpha s\int_{1-\varepsilon}^{s}\bar{\zeta}_{\varepsilon}(t) (1-t^2)^{\frac{\alpha-2}{2}}dt.
		\end{align*}
		For the integral on the right hand side, we find the following estimate.
		\begin{align*}
			&\alpha \int_{1-\varepsilon}^{s}\bar{\zeta}_{\varepsilon}(t) (1-t^2)^{\frac{\alpha-2}{2}}dt
			\leq \alpha \int_{1-\varepsilon}^{1} (1-t^2)^{\frac{\alpha-2}{2}}dt 
			 \leq \frac{\alpha}{1-\varepsilon} \int_{1-\varepsilon}^{1} t(1-t^2)^{\frac{\alpha-2}{2}}dt  \\
			&\qquad= \frac{(1-(1-\varepsilon)^2)^{\frac{\alpha}{2}}}{1-\varepsilon}
			\leq \frac{2^{\frac{\alpha}{2}}(1-(1-\tfrac{\varepsilon}{2})^2)^{\frac{\alpha}{2}}}{1-\varepsilon}
			\leq \frac{2^{\frac{\alpha}{2}}(1-s_0^2)^{\frac{\alpha}{2}}}{1-\varepsilon}.
		\end{align*}
		By combining these computations, we obtain
		\begin{align*}
			& \left| \alpha s\int_{0}^{s} \bar{\zeta}_{\varepsilon}(t) \bar{f}(t) (1-t^2)^{\frac{\alpha-2}{2}}dt \right|
			= \left| \alpha s\int_{1-\varepsilon}^{s} \bar{\zeta}_{\varepsilon}(t)\bar{f}(t) (1-t^2)^{\frac{\alpha-2}{2}}dt \right| \\
			&\qquad \leq \frac{2^{\frac{\alpha}{2}} }{1-\varepsilon}\abs{f(s_0)}(1-s_0^2)^{\frac{\alpha}{2}} 
			\leq \frac{2^{\frac{\alpha}{2}}}{1-\varepsilon} \sup_{x \geq 1-\varepsilon}\abs{\bar{f}(x)}(1-x^2)^{\frac{\alpha}{2}}. %\xrightarrow[\varepsilon\to0^+]{} 0
		\end{align*}
		The final expression is again independent of $s$ an converges to zero as $\varepsilon\to 0^+$.
		
		For $s\geq 1-\frac{\varepsilon}{2}$, we split the integral at $1-\frac{\varepsilon}{2}$, which yields
		\begin{align*}
			&\bigg|\int_{0}^{s} \bar{\zeta}_{\varepsilon}(t)\bar{f}(t) (1-t^2)^{\frac{\alpha-2}{2}}dt\bigg| \\
			&\qquad \leq \bigg|\int_{0}^{1-\frac{\varepsilon}{2}} \bar{\zeta}_{\varepsilon}(t)\bar{f}(t) (1-t^2)^{\frac{\alpha-2}{2}}dt\bigg| + \sup_{x\geq 1-\frac{\varepsilon}{2}} \bigg| \int_{1-\frac{\varepsilon}{2}}^{x} \bar{f}(t) (1-t^2)^{\frac{\alpha-2}{2}}dt \bigg|.
			%\xrightarrow[\varepsilon\to0^+]{} 0
		\end{align*}
		For the integral on $[0,1-\frac{\varepsilon}{2}]$, we can use our estimate from above to show that it tends to zero as $\varepsilon\to 0^+$. The final term is independent of $s$ and also tends to zero because the limit $\lim_{s\to 1}\int_0^s \bar{f}(t) (1-t^2)^{\frac{\alpha-2}{2}}dt$ exists.
		In conclusion, we have shown that $T_{\alpha}(\bar{\zeta}_{\varepsilon}\bar{f})(s)\to 0$ uniformly for $s\in[0,1]$. For $s\in [-1,0]$, the argument is completely analogous.
	\end{proof}

	\begin{prop}\label{prop:Tibij}
		For each $\alpha> 0$, the map $T_\alpha: \Dclass^\alpha \to C[-1,1]$ is a bijection and for $\bar{g}\in C[-1,1]$,
		\begin{equation*}
			T_{\alpha}^{-1}\bar{g}(t)
			= (1-t^2)^{-\frac{\alpha}{2}}\bar{g}(t) - \alpha t\int_0^t \bar{g}(s)(1-s^2)^{-\frac{\alpha+2}{2}}ds,
			\qquad t\in(-1,1).
		\end{equation*}
	\end{prop}
	\begin{proof}
		Clearly, $T_{\alpha}$ maps $C(-1,1)$ functions to $C(-1,1)$ functions and a direct computation verifies the inverse transform on the space $C(-1,1)$.
		We observe that $T_{\alpha}$ maps the subspace $\Dclass^\alpha$ into the subspace $C[-1,1]$. Thus, it remains to show that $T_{\alpha}^{-1}$ maps $C[-1,1]$ into $\Dclass^\alpha$.
		
		To this end, take some $\bar{g}\in C[-1,1]$ and let $\bar{f}:=T_{\alpha}^{-1}\bar{g}$. Then for all $t\in (-1,1)$,
		\begin{equation*}
			(1-t^2)^{\frac{\alpha}{2}}\bar{f}(t)
			= \bar{g}(t)-\alpha t (1-t^2)^{\frac{\alpha}{2}}\int_0^t \bar{g}(s)(1-s^2)^{-\frac{\alpha+2}{2}}ds.
		\end{equation*}
		Note that whenever $0<t_0<t<1$, by the mean value theorem for integrals, there exists $t_1\in (t_0,t)\subseteq (t_0,1)$ such that
		\begin{align*}
			&\alpha t (1-t^2)^{\frac{\alpha}{2}}\int_{t_0}^t \bar{g}(s) (1-s^2)^{-\frac{\alpha+2}{2}}ds
			= \alpha t (1-t^2)^{\frac{\alpha}{2}}\int_{t_0}^t \frac{\bar{g}(s)}{s} s(1-s^2)^{-\frac{\alpha+2}{2}}ds\\
			&\qquad= \frac{\bar{g}(t_1)}{t_1} \alpha t (1-t^2)^{\frac{\alpha}{2}}\int_{t_0}^t s(1-s^2)^{-\frac{\alpha+2}{2}}ds 
			 = \frac{\bar{g}(t_1)}{t_1} t \bigg(1 - \bigg(\frac{1-t^2}{1-t_0^2}\bigg)^{\!\!\!\frac{\alpha}{2}} \bigg).
		\end{align*}
		Let now $\varepsilon>0$ be arbitrary. Then, by the continuity of $\bar{g}$ at $t=1$, we may choose $t_0\in (0,1)$ such that for all $t,t_1\in (t_0,1)$,
		\begin{equation*}
			\left| \bar{g}(t) - \frac{\bar{g}(t_1)}{t_1} t \bigg(1 - \bigg(\frac{1-t^2}{1-t_0^2}\bigg)^{\!\!\!\frac{\alpha}{2}} \bigg) \right|
			< \varepsilon.
		\end{equation*}
		Consequently, with this choice of $t_0$, we obtain that
		\begin{align*}
			&\lim_{t\to 1} \left|  (1-t^2)^{\frac{\alpha}{2}}\bar{f}(t) \right|
			= \lim_{t\to 1} \left|  \bar{g}(t) - \alpha t (1-t^2)^{\frac{\alpha}{2}}\int_{0}^t \bar{g}(s) (1-s^2)^{-\frac{\alpha+2}{2}}dt \right| \\
			&\qquad = \lim_{t\to 1} \left|  \bar{g}(t) - \alpha t (1-t^2)^{\frac{\alpha}{2}}\int_{t_0}^t \bar{g}(s) (1-s^2)^{-\frac{\alpha+2}{2}}dt \right|
			\leq \varepsilon,
		\end{align*}
		where in the second equality, we used that changing the lower integral bound does not affect the limit.
		Since $\varepsilon>0$ was arbitrary, $\lim_{t\to 1}(1-t^2)^{\frac{\alpha}{2}}\bar{f}(t)=0$. The argument for the limit as $t\to -1$ is completely analogous.
		For the second condition, note that
		\begin{equation*}
			\alpha s\int_{0}^s \bar{f}(t) (1-t^2)^{\frac{\alpha-2}{2}}dt
			= \bar{g}(s) - (1-s^2)^{\frac{\alpha}{2}}\bar{f}(s),
		\end{equation*}
		so from what we have shown and the continuity of $\bar{g}$ at $\pm 1$, we obtain that the limits $\lim_{s \to \pm 1} \int_{0}^s \bar{f}(t) (1-t^2)^{\frac{\alpha-2}{2}}dt$ exist and are finite, and thus, $\bar{f}\in\Dclass^\alpha$.
	\end{proof}

	%%%%%%%%%%%%%%%%%%%%%%%%%%%%%%%%%%%%%%%%%%%%%%%%%%%%%%%%%%%%%%%%%%%%%%%%%%%
	\section*{Acknowledgments}
	%%%%%%%%%%%%%%%%%%%%%%%%%%%%%%%%%%%%%%%%%%%%%%%%%%%%%%%%%%%%%%%%%%%%%%%%%%%
	
	%The first- and third-named author were supported by the Austrian Science Fund (FWF), P31448-N35.
	The second-named author was supported by the Austrian Science Fund (FWF), \href{https://doi.org/10.55776/P34446}{doi:10.55776/P34446}.
	The third-named author was supported by the Austrian Science Fund (FWF), \href{https://doi.org/10.55776/ESP236}{doi:10.55776/ESP236}, ESPRIT program.
	
	The first- and third-named author thank the Hausdorff Research Institute for Mathematics in Bonn, where part of this work was realized during their stay at the Dual Trimester Program ``Synergies between modern probability, geometric analysis and stochastic geometry''.
	
	%%%%%%%%%%%%%%%%%%%%%%%%%%%%%%%%%%%%%%%%%%%%%%%%%%%%%%%%%%%%%%%%%%%%%%%%%%%

	\begingroup
	\let\itshape\upshape
	
	\bibliographystyle{abbrv}
	\bibliography{references}{}
	
	\endgroup

\end{document}